\documentclass[conference]{IEEEtran}

\usepackage{dsfont}
\usepackage[T1]{fontenc}
\usepackage{graphicx}
\usepackage{soul}
\usepackage{color}
\usepackage[utf8x]{inputenc}
\usepackage{amsfonts}
\usepackage{amsmath}
\usepackage{setspace}
\usepackage{amssymb}
\usepackage{booktabs,graphicx}
\usepackage[noadjust]{cite}
\usepackage{bbm}
\usepackage{epstopdf}
\usepackage{wasysym}
\usepackage{textcomp}

\usepackage{wrapfig}
\IEEEoverridecommandlockouts

\newcommand{\norm}[1]{\left\lVert#1\right\rVert}

\usepackage{cancel}
\usepackage{dblfloatfix}    
\usepackage{float}

\DeclareMathAlphabet{\mymathbb}{U}{BOONDOX-ds}{m}{n}

 \newtheorem{thrm}{Theorem}

\newtheorem{lem}{Lemma}

\newtheorem{assump}{Assumption}
\newtheorem{prop}{Proposition}

\newtheorem{rmk}{Remark}

\newcommand{\yl}[1]{{\color{black}{#1}}}
\newcommand{\yo}[1]{{\color{black}{#1}}}

\newcommand{\ir}[1]{{\color{black}{#1}}}
\newcommand{\irr}[1]{{\color{black}{#1}}}
\newcommand{\icr}[1]{{\color{black}{#1}}}

\newcommand{\ryo}[1]{{\color{black}{#1}}}

\newcommand{\il}{\textcolor{black}}
\newcommand{\ic}{\textcolor{black}}
\newcommand{\icc}{\textcolor{black}}
\newcommand{\icl}{\textcolor{black}}

\newcommand{\yrr}[1]{{\color{black}{#1}}}
\newcommand{\icf}[1]{{\color{black}{#1}}}

\usepackage[prependcaption,colorinlistoftodos]{todonotes}

\title{
A Distributed Scheme for Voltage and Frequency Control and Power Sharing in Inverter Based Microgrids}

\author{
	Yemi Ojo, Jeremy D. Watson, Khaled Laib and  Ioannis Lestas
\thanks{ 
	This work was supported by ERC starting grant 679774.}
\thanks{The authors are with the Department of Engineering, University of Cambridge, Trumpington Street, Cambridge, CB2 1PZ, United Kingdom. Emails: \{yo259, jdw69, kl507, icl20\} @cam.ac.uk.}
}

\begin{document}

	\maketitle
	\thispagestyle{empty}
	\pagestyle{empty}

	\begin{abstract}
	
Grid-forming \icf{inverter}-based autonomous microgrids present new operational challenges as the stabilizing rotational inertia of synchronous machines is absent. \ir{The design of efficient control policies for grid-forming inverters is, however, a non-trivial problem where multiple performance objectives need to be satisfied, including voltage/frequency regulation, current limiting capabilities, as well as \ryo{active} power sharing and a scalable operation.}
\yo{We propose in this paper a \ir{novel control architecture for frequency and voltage control which allows current limitation via an inner loop, \ryo{active} power sharing via a distributed secondary control policy and scalability by satisfying a passivity property.}
\ir{In particular,} the frequency controller employs the inverter output current and angle to provide an angle droop-like policy which improves its stability properties.  \ir{This also allows to incorporate a} secondary control policy for which we provide an analytical stability result which takes line conductances into account (in contrast to the lossless line assumptions in literature). The distinctive feature of the voltage control scheme is that it has a double loop structure that uses the DC voltage in the feedback control policy to implement a power-balancing strategy to improve performance.
\ir{The} performance of the control policy is illustrated via  simulations with detailed nonlinear models in a realistic setting.}

	\end{abstract}


	\section{INTRODUCTION}%
	
\label{intro}
The advancement in renewable energy technologies and increasing energy demand bring about the proliferation of renewable energy generation such as wind and solar.
These renewable energy resources are usually interfaced with inverters and deployed  as distributed generation (DG), in contrast to the centralized power grids where large synchronous machines (SMs)  are used.
The combination of DG units and network into a controllable system gives rise to  microgrids which can be operated in grid-connected or autonomous mode.
The latter relies on grid-forming inverters for  frequency and voltage regulation.
However, these present new operational challenges as the stabilizing rotational inertia of SMs is absent, and grid-forming inverters have inherently low-inertia~\cite{federicoDofler2018}.
Therefore, it becomes crucial to develop new approaches that guarantee the stability of autonomous microgrids.

\yo{The design of efficient control schemes for \ir{inverters with a grid forming role is a non-trivial problem due to the tight ratings of power electronics and the fast timescales of their dynamics, which leads to multiple objectives that need to be satisfied. In addition to voltage and frequency regulation, it is also important to be able to achieve \ryo{active} power-sharing at faster timescales. Furthermore, current limiting capabilities is a significant property that is often facilitated by means of double loop control architectures. Moreover, in order to allow a large scale integration of such inverters, it is important for the design to be scalable, i.e, stability is ensured via decentralized conditions that allow a plug-and-play operation. Existing schemes that have been proposed in the literature, have primarily focused on stability and each generally satisfies some of these objectives. Therefore the development of more advanced control policies with improved performance is a significant problem of practical relevance.}

In this paper, we present a novel control \ir{scheme} which \ir{aims} 
\yl{to} achieve \ir{the performance objectives described above.}
In particular,  we propose a control architecture for frequency and voltage control which is scalable\footnote{By being scalable we mean the control policy allows to ensure stability by satisfying decentralized stability conditions which allow \ic{a plug-and-play capability}, \yrr{i.e. new devices that satisfy the stability conditions  can be integrated into a network while maintaining its stability, thus allowing to extend the network to a much larger one.}}, allows current limitation via an inner loop, and leads naturally to a distributed secondary controller \ir{that} achieves \ryo{active} power sharing. The frequency controller employs the inverter output current and angle to provide an angle droop-like policy which improves its stability \ir{properties and leads to a secondary control policy. For the latter we} provide an analytical stability result which takes line conductances into account (in contrast to the lossless line assumptions \ir{often used in the} literature~\cite{simpson2013, andreasson2013distributed}). The distinctive feature of the voltage control scheme is that it has a double loop structure that uses the DC voltage in the feedback control policy to implement a power-balancing strategy to improve performance. Using passivity analysis, we are also able to guarantee the stability of the frequency and voltage control at faster time-scales.}

\icr{A preliminary version of this work appeared in conference paper \cite{yemiisgt2021}.
This extended manuscript includes detailed proofs, and additional simulations and discussion\footnote{
\icr{More precisely the additional material includes 
the detailed proof of Theorem \ref{stabilityps} and Lemma \ref{ml} on secondary control, more simulations providing a comparison with other control policies and more details in the analysis (Appendix~\ref{timesep} and Proposition~\ref{passproposition}).}}.}

\yo{\subsubsection*{\ir{Literature review}} \ir{Droop-based schemes have a simple implementation that does not require an additional communication layer,
~\cite{chandorkar1993, pogaku2007, arghirTaouba2018,  yemimo2020, yemijeremy2019, qzhongweiss2011}, however, they cannot provide stability guarantees in a scalable way and do not satisfy passivity properties, which are satisfied by our design. Angle droop control~(\cite{kolluri2017power, majumder2009angle, ritwikMajumder2010, ysunGuerrero2017}), on the other hand, 
does not achieve \ryo{active} power sharing at faster timescales.}
The  traditional  approaches \cite{pogaku2007, chandorkar1993, ysunGuerrero2017,  Rocabert2012} have used the current-error elimination method to provide current limiting capability in their inner (current) control loop. In contrast, our design uses the DC voltage in its inner control loop to implement a power-balancing strategy that improves performance by providing  current limiting capability and a tighter  DC  voltage regulation.

Non-droop alternatives \ir{have been proposed in the literature}, such as full-state feedback policies which use an open loop frequency control set by an internal oscillator, and many of these have plug-and-play capability. Voltage setpoints are sent by a centralized power management system and the inverters regulate their output voltage to this setpoint via a single loop. Examples include e.g.~\cite{sadabadi2016plug, tucci2020scalable, tucci2017voltage, riverso2014plug, tucci2016plug, jeremylestas2020, moradi2010robust}. However, in the period between setpoint updates, power sharing may not be guaranteed due to unexpected load changes, and these schemes consider only voltage control independently of the angle or frequency. By contrast, our scheme does not require voltage setpoints to be broadcast, \ir{and} is able to share power effectively even in the presence of unplanned load changes, \ir{by regulating both the} frequency and voltage simultaneously.
Another design was proposed \ir{in \cite{strehle2019port}} using a proportional controller in a port-Hamiltonian framework. \ir{This controller also} relies on the broadcast of accurate voltage setpoints and open loop frequency control is used. Furthermore, 
\ir{the aforementioned}
non-droop schemes are incompatible with double loop architectures, \ir{in the sense that they do not satisfy the passivity properties they rely upon when double loop control policies are introduced}. \ir{It should be noted} that single loop designs may not guarantee current limiting capabilities in the inverters, \ir{and the} usual industry practice is to 
\ir{achieve this via}
the current reference of the inner current loop.
Other recent designs exist which focus exclusively on the angle / frequency control, such as hybrid angle droop control \cite{tayyebi2020almost}, \cite{jouini2021inverse}. However, the setting differs considerably \ir{from the one considered in this paper} as voltage regulation is a key aim in the problem we consider.}

	\subsubsection*{Paper Contributions}
\yo{\ir{This paper addresses the problem \yl{of} control design for grid-forming inverters such that the following objectives are \yl{satisfied}: voltage/frequency regulation, \ryo{active} power sharing, current limiting capabilities, stability guarantees with plug-and-play operation. Its main contributions are summarized below:}
	\begin{enumerate}
	\item
	We propose a control architecture for frequency and voltage control which employs the inverter output current and angle to improve performance at fast time-scales. Furthermore, we ensure plug-and-play capability by satisfying a decentralized passivity condition.	
	\item Our control policy leads to a distributed secondary controller for which we provide an analytical stability result at slower time-scales with line conductances taken into account.
	\item  We propose an improved internal double loop structure that uses the DC voltage in the feedback control policy to implement a power-balancing strategy. Our policy provides current limiting capability and improved  DC  voltage regulation.
\end{enumerate}
\ir{Furthermore, a} case study using simulations
	with detailed  inverter models is used to demonstrate the desirable performance of the proposed controllers on an inverter-based microgrid. 
}

	\subsubsection*{Paper outline}
	The remainder of the paper is organised as follows. 	In section~\ref{model} we  present the microgrid  model.  
In section~\ref{PFC-a} we describe the frequency and voltage control schemes. A secondary control policy is proposed in section~\ref{powersharing}. Finally, simulation results are given in section~\ref{results}  and conclusions in
	 section~\ref{concl}.

	\subsubsection*{Notation}
	Let $\mathbb{R}_{\geq0}=\{x\in\mathbb{R}|x\geq0\}$,  $\mathbb{R}_{>0}=\{x\in\mathbb{R}|x>0\}$, and $\mathbb{S}=(-\frac{\pi}{2}, \frac{\pi}{2})$.
	We denote $\mathbf{1}_n$ ($\mathbf{0}_n$) the n-dimensional column vector of ones (zeros), $\mathbf{I}_{n}$ is the	identity matrix of size $n$, and $\mathbf{I}$ is used whenever dimension is clear from the  context.
	$\mathbf{0}_{n\times m}$ denotes an $n\times m$ zero matrix, and  $\mathbf{0}$ is used whenever dimension can be deduced from the  context.
	Let $\scriptsize{\mathbf{e}=[1~~0]^{\top}}$,
	$\scriptsize{\mathbf{e}_1=[0~~1]^{\top}}$, $\mathbf{e}_2=\scriptsize{\begin{bmatrix}0&1\\0&0 \end{bmatrix}}$, $J=\scriptsize{\begin{bmatrix}0&1\\-1&0 \end{bmatrix}}$, and $\mathbf{j}=\sqrt{-1}$.
	Let $x=\text{col}(x_1,\ldots,x_n) \in \mathbb{R}^{n}$ denote a column vector with entries $x_j \in \mathbb{R}$,
	and whenever clear from context we use the notation      $x=\text{col}(x_j)\in \mathbb{R}^{n}$. We denote $\text{diag}(a_j)\in\mathbb{R}^{n\times n}$,
	a 
{diagonal} matrix with diagonal entries $a_j$, $\text{blkdiag}(A_j)$ is a block diagonal  matrix with matrix  entries $A_j\in\mathbb{R}^{n\times n}$.
	The Kronecker product is denoted by $\otimes$, and  {for a matrix $A\in\mathbb{R}^{m\times n}$ we  denote its induced $2$-norm by $\|A\|_2$.} {For a Hermitian matrix ${G}\in\mathbb{C}^{n\times n}$ we denote its smallest eigenvalue by $\underline{\lambda}({G})$.}
	
We use the Park transformation 
to transform a balanced three-phase AC signal into its direct-quadrature components. The vector of such quantities at a bus $j$ in the local reference frame is found by using the local frequency $\omega_j(t)$ in the transformation, and we refer to this by the lower-case $dq$ subscript. Similarly, quantities in the common reference frame are found by using a constant common frequency $\omega_{0}$ in the transformation, and such quantities are referred to by the upper-case subscript $DQ$. The relationship between quantities in the $dq$ and $DQ$ frames is {given by:} 
	\begin{equation}	
	\begin{split}
	\label{alg4}
	x_{DQ}(t)=&{T}(\delta(t))x_{dq}(t),\\
	T(\delta(t)) =&{\begin{bmatrix}
	\cos \delta(t) &-\sin \delta(t)\\ \sin \delta(t) & \cos\delta(t)
	\end{bmatrix}}, \hspace{6mm}
	\dot\delta(t)=\omega(t)-\omega_{0},
	\end{split}
	\end{equation}
	where $\delta(t)\in\mathbb{S}$ is
	the angle between the  $dq$ and  $DQ$ reference frames.
 	$T(\delta(t))$ is a rotation matrix that satisfies the
	properties: \mbox{$T^{-1}(\delta(t))=T^{\top}(\delta(t))$}, \mbox{$\frac{\partial T(\delta(t))}{\partial\delta(t)}=J^{\top}T(\delta(t))$}.
	The time argument $t$ will often be omitted in the text for convenience in the presentation.

\section{Models and Preliminaries}\label{model}

\begin{figure}[t!]
	\centering
	\includegraphics[width=1\linewidth]{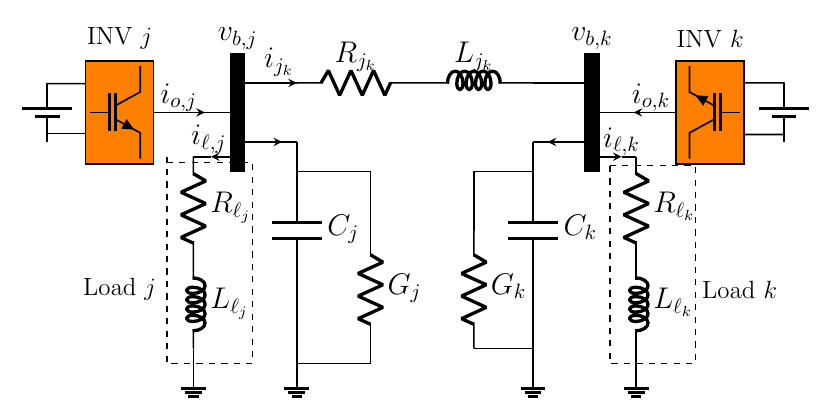}
		\vspace{-5.mm}
	\caption{Example of two inverters interconnected with a transmission line represented with a lumped $\pi$-model}.
	\vspace{-5mm}
	\label{lineloadinv}
\end{figure}

\subsection{Network model}\label{netw}
We describe the network model by a graph $(N,E)$ where $N=\{1, 2, \ldots, |N|\}$ is the set of buses, and $E\subseteq N\times N$ is the set of edges (power lines). The grid-forming inverters and loads are connected at the respective buses.
The entries of the incidence matrix $\mathcal{B} \in\mathbb{R}^{|N|\times |E|}$ are defined as  $\mathcal{B}_{jz}=1$ if bus $j$ is the source of edge $z$ and $\mathcal{B}_{jz}=-1$ if bus $j$ is the sink of edge $z$, with all other elements being zero.
 $\mathcal{L}=\mathcal{B}\mathcal{B}^{\top}\in\mathbb{R}^{|N|\times |N|}$ is the Laplacian matrix {of the graph}.
To present the physical model of the network which consists of the power lines and loads, we make the following assumption:
\begin{assump}
	All the power lines  are  symmetric three-phase lines, and the loads are balanced three-phase loads.
\end{assump}
Consider  the $\pi$-model of a line connecting  a bus $j\in N$ to a bus $k\in N$ with resistance and inductance  $R_{jk}, L_{jk} \in\mathbb{R}_{>0}$ and shunt capacitance and conductance $C_{j}, G_{j} \in\mathbb{R}_{>0}$ at corresponding buses; and a resistive-inductive (constant impedance) load with parameters  $R_{\ell_j}, L_{\ell_j} \in\mathbb{R}_{>0}$ (e.g. see Fig.~\ref{lineloadinv}).
The line and load models in the $DQ$ coordinates, rotating at the common reference frame frequency $\omega_{0}$, are easily derived by applying the Park {transformation} to the fundamental equations of the passive components, resulting in
 the line (equations~\eqref{net}) and  resistive-inductive load (equation~\eqref{netc}) models as follows:
\begin{equation}\small
	\label{net}
	\begin{split}
		C_l \dot V_{bDQ}&=(-G_l+\omega_0C_l\mathbf{J}) V_{bDQ}+ I_{oDQ} -I_{\ell DQ} - \mathbf{B} I_{lDQ}\\
		L_l \dot I_{lDQ}&=(-R_l+\omega_0L_l\mathbf{J})I_{lDQ}+\mathbf{B}^{\top}V_{bDQ}
	\end{split}
\end{equation}
\begin{equation}
	\label{netc}
	\hspace{-21mm}L_{\ell}\dot I_{\ell DQ}=(-R_{\ell}+\omega_0L_{\ell}\mathbf{J})I_{\ell DQ}+V_{bDQ}
\end{equation}
where $R_{l}=(\text{diag}(R_{jk})\otimes\mathbf{I}_2)$,
$L_{l}=(\text{diag}(L_{jk})\otimes\mathbf{I}_2)\in \mathbb{R}^{2|E|\times2|E|}$;
$C_{l}=(\text{diag}(C_{j})\otimes\mathbf{I}_2)$,
$G_{l}=(\text{diag}(G_{j})\otimes\mathbf{I}_2)$,
$R_{\ell}=(\text{diag}(R_{\ell_j})\otimes\mathbf{I}_2)$,
$L_{\ell}=(\text{diag}(L_{\ell_j})\otimes\mathbf{I}_2)$,
$\mathbf{J}=\text{blkdiag}(J)\in \mathbb{R}^{2|N|\times2|N|}$;
$\mathbf{B}=(\mathcal{B}\otimes\mathbf{I}_2)\in \mathbb{R}^{2|N|\times2|E|}$;
$I_{lDQ}=\text{col}(i_{DQ,jk})\in \mathbb{R}^{2|E|}$;
$V_{bDQ}=\text{col}(v_{bDQ,j})$,
$I_{oDQ}=\text{col}(i_{oDQ,j})$,
$I_{\ell DQ}=\text{col}(i_{\ell DQ,j})\in \mathbb{R}^{2|N|}$.
The  line current $i_{DQ,jk}=[i_{D,jk}~i_{Q,jk}]^{\top}$ takes values in $\mathbb{R}^2$; the injected current $i_{oDQ,j}=[i_{oD,j}~i_{oQ,j}]^{\top}$  at a bus $j\in N$  takes values in $\mathbb{R}^2$;
$i_{DQ,jk}$,  $i_{oDQ,j}$,   $i_{\ell DQ,j}$, $v_{bDQ,j}$ are two-dimensional vectors that include the $DQ$ components of the line current, injected current, load current and bus voltage respectively.

	\subsection{\yo{Grid-forming inverter model in common reference frame}}
	\label{gfim}
	Fig.~\ref{inverter} shows the schematic of a three-phase DC/AC grid-forming inverter.
	The DC circuit consists of a controllable current source $i_{dc,j}$ which takes values  in \ryo{$\mathbb{R}$}, a conductance $G_{dc_j}\in\mathbb{R}_{>0}$ and capacitance $C_{dc_j}\in\mathbb{R}_{>0}$. The AC circuit has an $LCL$ filter with  inductances
		$L_{f_j}, L_{c_j}\in\mathbb{R}_{>0}$, resistances $R_{f_j}, R_{c_j} \in\mathbb{R}_{>0}$, a conductance $G_{s_j} \in\mathbb{R}_{>0}$, and a shunt capacitance  $C_{f_j}\in\mathbb{R}_{>0}$.  $m_j$ is a balanced three-phase sinusoidal  control input signal, {used for the pulse-width modulation (PWM) that actuates the electronic switches}.
To present the physical model of the inverter, the following assumptions are made:
\begin{assump}\label{ass-average}
\begin{itemize}
	{
	\item The switching frequency is very high compared to the microgrid frequency and the filter sufficiently attenuates the harmonics.
	\item The power generated on the DC-side is transferred to the AC-side without switching losses.
}
\end{itemize}
\end{assump}
		From the assumptions above, the following can be used $i_{x,j}=\frac{1}{2}i_j^{\top}m_j$ and $v_{x,j}=\frac{1}{2}v_{dc,j}m_j$~\cite{yemimo2020, arghirTaouba2018}. These then allow to
		consider
	 the inverter model, formulated in the local ($dq$) reference frame, rotating with the local  frequency $\omega_j$  {(as in e.g. \cite{yemimo2020})}.
		\yo{The interconnection of multiple inverters usually results in multiple local $dq$ reference frames, which is due to the different local frequencies $\omega_j$ of the individual inverters.
\ir{This justifies
modeling} the inverters in a common reference frame.
		 In particular, we  interconnect the inverters with the network~\eqref{net} by transforming the inverter $dq$ model (as in \cite{yemimo2020})	to the common ($DQ$) reference frame, rotating at \ir{a} constant common frequency $\omega_{0}$.}
		Let the  variable $m_{dq,j}$ ($m_{DQ,j}$) denote
		the two-dimensional
		$dq$ ($DQ$) coordinates of the control input variable $m_j$ of inverter $j$, and $m_{dq}=\text{col}(m_{dq,j})$, $m_{DQ}=\text{col}(m_{DQ,j})$.
		Using~\eqref{alg4}, the representation of the $dq$ model
		 in the $DQ$ frame is compactly given for the multi-inverter model,
		 with	${m}_{DQ}=\mathbf{T}(\delta){m}_{dq}$,
		 $\text{col}(i_{xDQ,j})=\frac{1}{2}\mathbf{I}^{\top}_{DQ}m_{DQ}$, $\text{col}(v_{xDQ,j})=\frac{1}{2}\mathbf{V}_{dc}m_{DQ}$, as
		\begin{figure}[t!]
			\centering
			\includegraphics[width=1\linewidth]{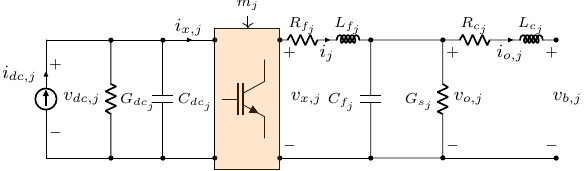}
				\vspace{-6.mm}
			\caption{Grid-forming inverter circuit diagram.}
			\vspace{-4mm}
			\label{inverter}
		\end{figure}	
		\begin{subequations}
		\label{crfinv2}
		\begin{align}
		\label{crfinv2a}
		\dot \delta =& \omega- \omega_0\mathbf{1}_n\\
		\label{crfinv2b}
		C_{dc}\dot V_{dc}=&-G_{dc}V_{dc}+I_{dc}-\frac{1}{2}\mathbf{I}^{\top}_{DQ}m_{DQ}\\
		L_{f}\dot  I_{DQ}=&(-R_{f}+\omega_0 L_{f}\mathbf{J})I_{DQ}+\frac{1}{2}\mathbf{V}_{dc}m_{DQ}-V_{oDQ}\\
		C_{f}\dot V_{oDQ}=&(-G_{s}+\omega_0 C_{f}\mathbf{J})V_{oDQ}+I_{DQ}-I_{oDQ}\\
		L_{c} \dot I_{oDQ}=&(-R_{c}+\omega_0 L_{c}\mathbf{J})I_{oDQ}+V_{oDQ}-V_{bDQ}
		\end{align}
		\end{subequations}
	where
	$\omega=\text{col}(\omega_{j})$,
	$I_{dc}=\text{col}(i_{dc,j})$,
	$V_{dc}=\text{col}(v_{dc,j}) \in \mathbb{R}^{|N|}$;
	$\delta=\text{col}(\delta_{j})\in \mathbb{S}^{|N|}$;
	$I_{DQ}=\text{col}(i_{DQ,j})$,
	$V_{bDQ}=\text{col}(v_{bDQ,j})$,
	$I_{oDQ}=\text{col}(i_{oDQ,j})$,
	$m_{dq}=\text{col}(m_{dq,j})
	\in \mathbb{R}^{2|N|}$;
	$C_{dc}=\text{diag}(C_{dc_j})$, $G_{dc}=\text{diag}(G_{dc_j})\in \mathbb{R}^{|N|\times|N|}$;
	$R_{f}=(\text{diag}(R_{f_j})\otimes\mathbf{I}_2)$,
	$R_{c}=(\text{diag}(R_{c_j})\otimes\mathbf{I}_2)$,
	$L_{f}=(\text{diag}(L_{f_j})\otimes\mathbf{I}_2)$,
	$L_{c}=(\text{diag}(L_{c_j})\otimes\mathbf{I}_2)$,
	$C_{f}=(\text{diag}(C_{f_j})\otimes\mathbf{I}_2)$,
	$G_{s}=(\text{diag}(G_{s_j})\otimes\mathbf{I}_2)$,
	$\mathbf{V}_{dc}=(\text{diag}(v_{dc,j})\otimes\mathbf{I}_2)$,
	$\mathbf{T}(\delta)=\text{blkdiag}(T(\delta_j))\in \mathbb{R}^{2|N|\times2|N|}$;
	$\mathbf{I}_{DQ}=(\text{diag}(i_{D,j})\otimes\mathbf{{e}}+\text{diag}(i_{Q,j})\otimes\mathbf{{e}_1}) \in \mathbb{R}^{2|N|\times|N|}$;
	${m}_{DQ}=\text{col}({m}_{DQ,j}) \in\mathbb{R}^{2|N|}$,
	$n=|N|$.
	$i_{DQ,j}$, $i_{oDQ,j}$, $v_{DQ,j}$, $v_{oDQ,j}$
	are two-dimensional vectors that include the $DQ$ components of the inverter currents and voltages respectively.


\subsection{Passivity}\label{passivity}
We review {in this section} the notion of passivity
	and its use to   guarantee microgrid stability in a decentralized way.
	We use the notions of passivity and strict passivity as defined in~\cite[Definition $6.3$]{khalil2014},
	but with \irr{the state, input and output} $x,u,\yo{y}$ replaced by  the deviations $x-x^*$, $u-u^*$, \yo{$y-y^*$} respectively, \irr{where $x^\ast, u^\ast, y^\ast$ are values at an equilibrium point.} \irr{Furthermore, we say a system is (strictly) passive about an equilibrium point $x^\ast, u^\ast$ if the condition on the storage function in the passivity definition holds for all values of $x, u$ in some neighbourhoods of $x^\ast, u^\ast$ respectively.}
	The negative feedback interconnection of two passive systems is stable and passive~\cite{khalil2014}. Hence, by representing the microgrid as a negative feedback interconnection of two passive subsystems, its closed-loop stability  can be guaranteed in a decentralized manner. To this end, we decompose the microgrid into two subsystems, namely the network, \irr{which includes the line dynamics, and  the inverter dynamics, as illustrated} in Fig.~\ref{fdback}.
	The network dynamics ~\eqref{net},~\eqref{netc}
\irr{have as} output $V_{bDQ}$ and input $I_{oDQ}$, while the inverter dynamics \eqref{crfinv2}
\irr{have as} output $I_{oDQ}$ and input $-V_{bDQ}$.
	By exploiting the passivity property of the network  when this is represented in $DQ$ coordinates, stated for completeness in Theorem~\ref{thrmnet}, it can be shown that Assumption~\ref{asspas} is a sufficient {\it decentralized} condition for stability, as stated in Theorem~\ref{thrmclp} (see e.g.~\cite{jeremylestas2020} where a more advanced line model is also used).
	The proofs of  Theorem~\ref{thrmnet},~\ref{thrmclp} are analogous to those in \ryo{{e.g.}~\cite{jeremylestas2020}.}

\begin{thrm}[Passivity of network in $DQ$ frame] \label{thrmnet}
	Suppose there exist an equilibrium point $x^*_N=[I^{*\top}_{l DQ},I^{*\top}_{\ell DQ},V^{*\top}_{bDQ}]^\top$, with 
input ${u}^*=I^*_{oDQ}$ \yo{and output $y^*= V^*_{bDQ}$,} then the network~\eqref{net},~\eqref{netc} with input ${u}=I_{oDQ}$ and output ${y}= V_{bDQ}$ is strictly passive about the equilibrium\footnote{{It should be noted that since the network model $\eqref{net},~\eqref{netc}$ that includes the line dynamics is linear, the passivity property in Theorem \ref{thrmnet} holds about any equilibrium point, \irr{and also for any deviation from the equilibrium point}.
}} $(x^*_N,u^*)$.
\end{thrm}
	\begin{rmk}	
		We have considered constant impedance  loads in the network which are known to be passive.
Constant power loads can be  nonpassive 
due to  their negative incremental resistance~\cite{ashourloo2013stabilization}.
		In the latter case, the network is guaranteed to be passive under an appropriate condition as derived in~\cite{strehle2019port}, that is satisfied 
when a sufficient number of constant impedance loads is present.
	\end{rmk}

\begin{assump}\label{asspas}
	Each
	inverter in the system~\eqref{crfinv2}
	with state vector $x=[\delta^\top, V_{dc}^\top, I^\top_{DQ}, V^\top_{oDQ},$ $ I_{oDQ}]^\top$, input
	$u=-V_{bDQ}$ and output $y=I_{oDQ}$	
	satisfies the {strict} passivity property {in}
	\cite[Definition $6.3$]{khalil2014}
	about an equilibrium point $(x^*,u^*)$.
\end{assump}

\begin{thrm}[Closed-loop stability]\label{thrmclp}
	Suppose there exists an equilibrium point 	$x^*_m=(x^*_N,x^*)$ of the interconnected inverter dynamics~\eqref{crfinv2}
	and the network~\eqref{net},~\eqref{netc}, for which the inverter dynamics satisfy Assumption~\ref{asspas} for all $j\in N$. Then such an  equilibrium point is asymptotically
stable.
\end{thrm}

\begin{rmk}
	The advantage of the stability criterion in Assumption~\ref{asspas} is that it is a 
decentralized condition.
	Since the network is passive in the $DQ$ frame, in the remainder of the paper we aim to passivate the inverter system  via an appropriate control policy. As mentioned in the introduction a distinctive feature of the proposed policy is the double loop architecture that uses the DC voltage in the feedback control loop to implement
	a power-balancing strategy to improve performance,	
	and its ability to incorporate a distributed secondary control schemes for \ryo{active} power sharing.
\end{rmk}

\begin{figure}[t!]
	\centering
	\includegraphics[width=1\linewidth]{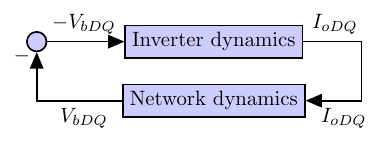}
	\vspace{-5mm}
	\caption{
		Negative feedback connection of inverter and network dynamics.}
	\label{fdback}
	\vspace{-4mm}
\end{figure}

\section{Proposed  Control Schemes}\label{PFC-a}
In this section, we  present a control architecture for frequency and voltage control  that guarantees stability by making the inverters 
to satisfy the  {passivity property in  Assumption~\ref{asspas}} through a local 
design.

\subsection{Proposed frequency control} \label{PFC}
Grid-forming inverters must operate in a synchronized manner despite load  variations and system uncertainties.	
Our aim is to design a decentralized frequency control scheme  \ryo{that} restores the  frequency to  its nominal value after a disturbance, and can be incorporated with a secondary control policy to provide \ryo{active} power sharing capabilities.

%

	To this end, we propose a frequency control scheme
	that can be seen as an improved angle droop policy, which  leads to passivity properties in the $DQ$ frame.
	This scheme takes the inverter output current
	$i_{oD,j}:={\mathbf{e}}^{\top}i_{oDQ,j}$
	(i.e., the first component of  $i_{oDQ,j}$)
	and the		angle $\delta_j$  as feedback  to adapt the frequency as described below:
	\begin{equation}
		\label{derr8-6}
		\omega_j= \omega_0-k_{p,j}{\mathbf{e}}^{\top}i_{oDQ,j} - k_{I,j} \delta_j +\chi_j
	\end{equation}
	where $k_{p,j}, k_{I,j}\in\mathbb{R}_{>0}$ are the  droop and damping gains respectively,  and $\chi_j\in\mathbb{R}$ are set-points.
	Let $I_{oDQ}=\text{col}(i_{oDQ,j})\in\mathbb{R}^{2|N|}$, $k_{p}=\text{diag}(k_{p,j}), k_{I}=\text{diag}(k_{I,j}) \in\mathbb{R}^{|N|\times |N|}_{>0}$, $\underline{\mathbf{e}}=(\mathbf{I}_n\otimes\mathbf{e})\in\mathbb{R}^{2|N|\times |N|}$,  $\chi=\text{col}(\chi_j)\in\mathbb{R}^{|N|}$, $n=|N|$. The compact form of \eqref{derr8-6} for multiple inverters is given by
	\begin{equation}
		\label{der8}
		\omega= \omega_0\mathbf{1}_{n}-k_{p}\underline{\mathbf{e}}^{\top}I_{oDQ} - k_{I} \delta +\chi.
	\end{equation}

		Considering~\eqref{der8} with the 
\ryo{angle} {$\delta$} dynamics~\eqref{crfinv2a}  results in an improved version of angle droop
		where the term $k_{I}\delta$ provides the necessary damping of the angle \ryo{dynamics,}
		which helps the inverter model~\eqref{der8} to satisfy a passivity property in the $DQ$ frame (discussed in sections~\ref{main},~\ref{resultspass}).
		As it will be discussed in section~\ref{powersharing}, the  current $I_{oD}$ allows to achieve \ryo{active} power sharing by appropriately adjusting  $\chi$, and the choice of $k_p$ sets the power sharing ratio.
		The parameters $\chi$ are assumed to be transmitted to each	inverter by a high-level control policy on a slower timescale, i.e., typically a secondary control, or energy management system \ic{as the effect of the clock drifts \icc{is small} at faster timescales.}
		 $\chi$ can provide additional capability to correct clock drifts which may arise due to clock inaccuracies as discussed in~\cite{kolluri2017power}.
		 \ryo{Also in this case the updating of  $\chi$ can involve the use \ic{of GPS}\footnote{Global Positioning System.} \cite{tucci2020scalable}, \cite{singh2011applications}.}	
		Furthermore, \eqref{der8} with
		~\eqref{crfinv2a}
		ensures that
		the equilibrium frequency of each inverter is equal to the common constant frequency, i.e.
  \mbox{$\omega^*=\omega_{0}\mathds{1}_n$}  (section~\ref{results}).

		It is informative to compare our proposed controller~\eqref{der8} with
		\eqref{crfinv2a} to \ryo{the} traditional angle droop control~\cite{kolluri2017power, majumder2009angle, ritwikMajumder2010, ysunGuerrero2017}.
		One of the advantages of our proposed control scheme is scalability, which is achieved via satisfying an appropriate passivity property as mentioned above. 	
		The use of $I_{oD}$ in~\eqref{der8} helps to avoid the nonlinearity associated with the active power relation used in traditional droop control schemes~\cite{chandorkar1993, kolluri2017power, majumder2009angle, ritwikMajumder2010, ysunGuerrero2017}.
		A further benefit is that~\eqref{der8}
		 with
		~\eqref{crfinv2a} provides inertia and damping similar to the dynamic behaviour of the SM, which is not achievable with traditional angle droop control~\cite{kolluri2017power, majumder2009angle, ritwikMajumder2010, ysunGuerrero2017}.
		To show this,	substitute~\eqref{der8} into  the \ryo{angle $\delta$} 
		 dynamics~\eqref{crfinv2a}, and
		expressed in \ryo{a} more insightful form gives
	\begin{equation}
	\label{iner1}
	M\dot\delta=  - D \delta -\underline{\mathbf{e}}^{\top}I_{oDQ} +M\chi_j,
\end{equation}
where  $M=k^{-1}_{p}$,  $D=k^{-1}_{p}k_{I}$.
Equation~\eqref{iner1} is analogous to a swing equation, with the frequency replaced by the angle $\delta$.
$M$ corresponds to the inertia, and $D$ the damping coefficient. The droop gain $k_p$  can be chosen to shape the desired (virtual) inertia  $M$, and $k_{I}$ provides an additional degree of freedom to design $D$.
This is an improvement compared to the traditional angle droop control~\cite{kolluri2017power, majumder2009angle, ritwikMajumder2010, ysunGuerrero2017} where the inertia $M$ is zero and only $k_p$  is available to design $D$.

		\subsection{DC voltage regulation}
		\label{dcreg}
		It is desirable that the DC voltage is regulated to a predefined setpoint.
		Hence we present a DC voltage proportional-integral (PI) controller that achieves this as follows:
	\ryo{\begin{equation}	
		\label{idc2}
		\begin{split}
			\dot \zeta=&V_{dc}-V_{dc,r}\\
			I_{dc}=&-\Lambda_{P}(V_{dc}-V_{dc,r})-\Lambda_{I}\zeta,	
		\end{split}
	\end{equation}}
		where 		$\zeta=\text{col}(\zeta_{j})\in\mathbb{R}^{|N|}$ is the integrator state,
		$\Lambda_{P}=\text{diag}(\Lambda_{P,j})$,
		$\Lambda_{I}=\text{diag}(\Lambda_{I,j})\in\mathbb{R}^{|N|\times|N|}_{>0}$ are the DC proportional and integral gains respectively, $V_{dc,r}=\mathbf{1}_n v_{dc,r},\,\ryo{V_{dc,r}\in\mathbb{R}^{|N|}_{>0}}$, and  $v_{dc,r}\in\mathbb{R}_{>0}$ denotes the DC voltage setpoint.

	\subsection{Inverter output voltage regulation}\label{acreg}
		Grid-forming inverters are required to regulate the voltage of the grid they form, hence they need to have  voltage regulation capability.
This is achieved in our proposed scheme via the control signal $m_{DQ,j}$ in~\eqref{crfinv2}.
		\yo{In particular, we use the double (outer and inner) loop design where the inner loop is faster than the outer one.
		One of the distinctive features of our scheme is that it uses the DC voltage
		in the inner control loop and incorporates the angles while providing voltage control.
	 Our control scheme  \ir{described in detail below} \ic{(also illustrated with block diagrams in section \ref{sec:impl}, Fig. \ref{blockfig}, \ref{twoloops}, \ref{dcloop})}.}

			First, \yo{a} reference current $i_{DQ,j}^{r}$
			is generated by the outer voltage loop  by means of PI control acting on the voltage deviation $v_{oDQ,j}-{T}(\delta_j)\mathbf{{e}}V_n-n_{q,j}\mathbf{e}_2i_{oDQ,j}$, where $n_{q,j}, V_n\in\mathbb{R}_{>0}$ are the voltage droop gain and  nominal voltage respectively.
			We use $-n_{q,j}\mathbf{e}_2i_{oDQ,j}=-n_{q,j}i_{oQ,j}$ to adjust the direct-coordinate of ${T}(\delta_j)\mathbf{{e}}V_n$, similar to the conventional reactive power based voltage droop control in~\cite{pogaku2007, chandorkar1993}. \yo{\ir{Hence the} voltage loop incorporates the angle.}
		Then, the inner control loop generates  $m_{DQ,j}$ by means of PI control acting on the power imbalance $i_{DQ,j}v_{dc,r}-i^r_{DQ,j}v_{dc,j}$.
		 		The voltage is therefore described by
		\begin{subequations}
		\label{acv3}
		\begin{align}
			\label{acv3-a}
		\dot \beta_{DQ,j}=&v_{oDQ,j}-{T}(\delta_j)\mathbf{{e}}V_n-n_{q,j}\mathbf{e}_2i_{oDQ,j}\\
		\notag
		i^r_{DQ,j}=&-c_{p,j}\left(v_{oDQ,j}-{T}(\delta_j)\mathbf{{e}}V_n-n_{q,j}\mathbf{e}_2i_{oDQ,j}\right)
		\\ \label{acv3-b}&-c_{I,j}\beta_{DQ,j}\\
		\label{acv3-c}
		\dot\xi_{DQ,j}=&i_{DQ,j}v_{dc,r}-i^r_{DQ,j}v_{dc,j}\\
		\label{acv3-d}
		\hspace{-.8mm}m_{DQ,j}=&-\lambda_{P,j}(i_{DQ,j}v_{dc,r}-i^r_{DQ,j}v_{dc,j})-\lambda_{I,j}\xi_{DQ,j}
		\end{align}
		\end{subequations}	
where $\beta_{DQ,j}$, $\xi_{DQ,j}$ 	are two-dimensional vectors that include the $DQ$ components of the the respective integrator states of the voltage and \yo{inner} control loops; $c_{p,j}, \lambda_{P,j}, c_{I,j}, \lambda_{I,j}\in\mathbb{R}_{>0}$ are the respective control loops proportional and integral gains.
	\begin{rmk}\label{rmk-volreg-motn}
	\ryo{We note that the use of $i_{DQ,j}v_{dc,r} - i_{DQ,j}^r v_{dc}$  enhances the passivity property of the inverter system. This is motivated \il{by \ic{a passivity} analysis of system \eqref{crfinv2}, \eqref{der8}, \eqref{idc2},} which we omit here for the readability of the text.}
\end{rmk}

\ryo{~We now  present the compact form  of \eqref{acv3} for multiple inverters.}
Let $\beta_{DQ}=\text{col}(\beta_{DQ,j}), \xi_{DQ}=\text{col}(\xi_{DQ,j}),
{I}^r_{DQ}=\text{col}(i^r_{DQ,j})		\in\mathbb{R}^{2|N|}$;
$c_{p}=(\text{diag}(c_{p,j})\otimes\mathbf{I}_2)$, $\lambda_{P}=(\text{diag}(\lambda_{P,j})\otimes\mathbf{I}_2)$,
$c_{I}=(\text{diag}(c_{I,j})\otimes\mathbf{I}_2)$,
$\lambda_{I}=(\text{diag}(\lambda_{I,j})\otimes\mathbf{I}_2)$,
$\underline{n}_q=\text{blkdiag}(\mathbf{e}_2n_{q,j})\in\mathbb{R}^{2|N|\times 2|N|}_{>0}$;
$\mathbf{I}^r_{DQ}=(\text{diag}(i^r_{D,j})\otimes\mathbf{{e}}+\text{diag}(i^r_{Q,j})\otimes\mathbf{{e}_1}) \in \mathbb{R}^{2|N|\times|N|}$.
The  compact form of the voltage control scheme \eqref{acv3} is given by:		\begin{subequations}	\label{acv1}
	\begin{align}
		\label{acv1a}
		\dot \beta_{DQ}=&V_{oDQ}-\mathbf{T}(\delta)\mathbf{\underline{e}}V_n-\underline{n}_qI_{oDQ}\\
			\label{acv1b}
		I^r_{DQ}=&-c_{p}(V_{oDQ}-\mathbf{T}(\delta)\mathbf{\underline{e}}V_n-\underline{n}_qI_{oDQ})-c_I\beta_{DQ}\\
			\label{acv1c}
		\dot\xi_{DQ}=&\mathbf{I}_{DQ}V_{dc,r}-\mathbf{I}^r_{DQ}{V}_{dc}\\
			\label{acv1d}
		{m}_{DQ}=&-\lambda_{P}(\mathbf{I}_{DQ}V_{dc,r}-\mathbf{I}^r_{DQ}{V}_{dc})-\lambda_{I}\xi_{DQ}.
	\end{align}
\end{subequations}
	
\yo{Our voltage control policy \eqref{acv1} differs from existing control schemes in \ir{its} implementation and the advantages it offers.
	One of the distinctive features of our scheme is that it uses the DC voltage
	in the inner control loop \eqref{acv1c}-\eqref{acv1d}, and thus a power-balancing strategy is implemented which  improves DC  voltage regulation and offers current limiting capability.
	This differs to the conventional approach \cite{pogaku2007, chandorkar1993, ysunGuerrero2017,  Rocabert2012} where \yl{the} current-error
is eliminated to provide current limiting capability in their inner (current) control loop.
Also, our control architecture incorporates angle dynamics   into its outer  voltage loop \eqref{acv1a}-\eqref{acv1b} while voltage control is achieved, in contrast to
\cite{sadabadi2016plug, tucci2020scalable, tucci2017voltage, riverso2014plug, tucci2016plug, strehle2019port, strehle2021unified, jeremylestas2020} which implement
voltage control independent of the angle or frequency.}
\ryo{Moreover, the} use of $I_{oQ}$ helps to avoid  the nonlinearity associated with  the reactive power relation used in the conventional voltage droop control in~\cite{chandorkar1993, pogaku2007}.

\ryo{As already mentioned above, our voltage control through its power-balancing strategy  offers \ic{a current} limiting capability. Note that in  \eqref{acv1}, the integral action of its inner control loop forces   $\mathbf{I}_{DQ}V_{dc,r}$   to track    $\mathbf{I}^r_{DQ}V_{dc}$.
With  	\eqref{idc2} implemented, the integral action of the inner control loop would force  $\mathbf{I}_{DQ}$  to track $\mathbf{I}^r_{DQ}$.
\il{Since $\mathbf{I}^r_{DQ}$ depends on the voltage deviations that are in general small, $\mathbf{I}^r_{DQ}$ and hence also $\mathbf{I}_{DQ}$ do not have large fluctuations.	It should be noted that this is the} usual industry practice of how double loop control architectures achieve  current limiting capabilities \cite{pogaku2007, Rocabert2012}.}

\subsection{Passivity of inverter system}
\label{main}

The passivity analysis of the inverter system ~\eqref{crfinv2},~\eqref{der8},~\eqref{idc2}, ~\eqref{acv1} is performed \irr{for the dynamics relevant} at a fast timescale. Thus the secondary control parameter {$\chi$} is taken as constant since this is adjusted at  slower timescales.
Defining the state vector $x=[\delta^\top\, V_{dc}^\top\, I^\top_{DQ}\, V^\top_{oDQ}\, I_{oDQ}^{\top}\, $ $\beta^{\top}_{DQ}\, \xi^{\top}_{DQ}]^\top$ and the deviations $\tilde{x}=x-x^*$, 
{${\tilde{V}_{bDQ}=}{V}_{bDQ}-{V}^*_{bDQ}$, ${\tilde{I}_{oDQ}=}{I}_{oDQ}-{I}^*_{oDQ}$},
the linearization  of the inverter system ~\eqref{crfinv2},~\eqref{der8},~\eqref{idc2},~\eqref{acv1}   about  $(x^*, \omega_{0}, I^*_{dc}, m^*_{DQ}, {V}^*_{bDQ})$ {with $\tilde{u}=\tilde{V}_{bDQ}$, $\tilde{y}=\tilde{I}_{oDQ}$} is 
\begin{equation}
	\label{linmodell}
		\begin{split}
\dot{\tilde{x}}&=(A-C_{\delta}^{\top}k_p\underline{\mathbf{e}}^{\top}C-C_{\delta}^{\top}k_IC_{\delta}-B\hat K)\tilde{x}+B_u\tilde{u}, \hspace{4mm} \\
	\tilde{y}&=C\tilde{x}+D_u\tilde{u},	
		\end{split}
\end{equation}
{where $A, B, B_u,  C, C_\delta, D_u, \hat{K}$ are given in Appendix~\ref{Amat}.}

{We state Proposition~\ref{passproposition}, which follows directly from the KYP Lemma~\cite{khalil2014}. 
Proposition~\ref{passproposition} gives  a gain selection criterion that allows to  choose {appropriate}
$k_p, k_I, \underline{n}_q, c_{p}, c_{I},  \lambda_{P}, \lambda_{I}$
such that each inverter in~\eqref{linmodell} satisfies {the passivity property {in Assumption~\ref{asspas},}
} 
which is {a decentralized} condition for {stability.} 

\begin{prop}
\label{passproposition}
Consider the inverter system~\eqref{linmodell} with input $\tilde{u}=\tilde{V}_{bDQ}$ and output  $\tilde{y}=\tilde{I}_{oDQ}$.
The inverter system is 
strictly passive\footnote{\yrr{\icf{It should be noted that since Proposition~\ref{passproposition} is associated with the linearized system ~\eqref{linmodell}, 
the  corresponding stability result in Theorem \ref{thrmclp} would be local for the original nonlinear system.}}}
if there exists a positive definite matrix $P=P^{\top}$ {and some $\epsilon>0$} such that
\begin{equation}
	\begin{bmatrix}\label{lmii}
		\Sigma {\,+\,\epsilon P}    & PB_{u}- C^{\top}  \\
		 B^{\top}_{u}P-C    & -D^\top_u-D_u
	\end{bmatrix}\leq0
\end{equation}
where
{$\Sigma$ is  defined in Appendix \ref{Amat}.}
\end{prop}
{The proof follows directly from the KYP Lemma \cite{khalil2014}.}

\begin{rmk}\label{controllergainchoice}
A possible approach to tune the {controller} parameters
is to first choose  $k_p, \underline{n}_q, c_{p}, c_{I},  \lambda_{P}, \lambda_{I}$ 
and then adjust $k_{I}$ so that  the passivity condition \eqref{lmii} is satisfied (discussed in more detail in section \ref{resultspass}).
This approach was followed in various benchmark examples discussed in section  \ref{resultspass} where the  passivity property is satisfied.
\end{rmk}

\begin{rmk}\label{rmk:kyp}
	{{An alternative} way to verify the  passivity property in Assumption 3 is via the strict positive realness of the inverter system~\eqref{linmodell} transfer function $G(s)=C(s\mathbf{I}-A+C_{\delta}^{\top}k_p\underline{\mathbf{e}}^{\top}C+C_{\delta}^{\top}k_IC_{\delta}+B\hat K)^{-1}B_u+D_u $  \cite[Lemma 6.1]{khalil2014}. In particular,
 the eigenvalues of the Hermitian part of the transfer matrix at all frequencies must be {positive, 
i.e.  $\scriptsize{G(\mathbf{j}\omega)+G^{*}(\mathbf{j}\omega)>0}$}.
}
\end{rmk}

\section{Secondary Control Scheme}
		\label{powersharing}
		Here we discuss the active power sharing that~\eqref{der8} can  provide when
		 $\chi$ is updated via the distributed scheme described below, which can be seen as a secondary control policy occurring at slower timescales:
		\begin{equation}\label{sec1}
		\dot \chi = -\alpha\mathcal{L}\chi+\alpha\mathcal{L}k_I\delta.
		\end{equation}
		where  $\alpha>0$.
		\yo{In contrast to \ir{other} non-droop strategies
			e.g.~\cite{zhang2014online, zhang2016transient, sadabadi2016plug, tucci2020scalable, tucci2017voltage, riverso2014plug, tucci2016plug, jeremylestas2020, moradi2010robust},
incorporating the secondary control \eqref{sec1} \ir{within} our policy \eqref{der8} \ir{allows to achieve \ryo{active} power sharing without having to rely on setpoint updates via
	\ryo{optimal power flow solutions.} This allows \ryo{active} power sharing to be achieved in the presence of  unexpected load changes.

\ir{The \ryo{active} power sharing property \ryo{(Remark \ref{rmk-power})} achieved at equilibrium by our control policy \eqref{sec1}  follows from the \ryo{current sharing} relation stated below \ryo{in Proposition \ref{pshare}.}}}}
		\begin{prop}[\ryo{Current } sharing] \label{pshare}
			At equilibrium the dynamics given by~\eqref{net},~\eqref{crfinv2},~\eqref{der8},~\eqref{idc2},~\eqref{acv1},~\eqref{sec1} satisfy
			\begin{equation}\label{condp}
			\frac{i^*_{oD,j}}{i^*_{oD,k}}=\frac{k_{p,k}}{k_{p,j}}, \hspace{10mm} \forall j,k\in N.
			\end{equation}
			\end{prop}
			The proof can be found  in Appendix~\ref{prf:pshare}.
			\begin{rmk}[\ryo{Power sharing}]\label{rmk-power}
				Note that~\eqref{condp} gives approximate power sharing at equilibrium. 					
				\ryo{\ic{In particular, under} the assumption that the quadrature-component \il{of the voltage $v^*_{oQ,j}$ is} significantly smaller compared to the \il{direct component $v^*_{oD,j}$} \il{(discussed in Remark \ref{rem:Papp})},   the active power relation can be \il{approximated as}
				$P^*_{o,j}:=v^*_{oD,j}i^*_{oD,j},\,\forall j\in N$. Using \il{also}
					\eqref{condp}, the power sharing ratio  between inverter $j,k\in N$ is \il{then} given by}	
								\begin{equation}	\label{condp1}
									\frac{P^*_{o,j}}{P^*_{o,k}}=\frac{v^*_{oD,j}}{v^*_{oD,k}} ~\frac{k_{p,k}}{k_{p,j}} , \hspace{3mm} \forall j,k\in N.
								\end{equation}				
					If $v^*_{oD,j}=v^*_{oD,k}$, which is a property that approximately holds since the voltage\footnote{\yrr{In our analysis we have taken the line resistances into account which would cause voltage drop. Nonetheless, the design and implementation of our voltage control mechanism in each inverter helps to keep the voltage deviations  generally small. We have demonstrated this in  Fig. \ref{figs}(e) where in the simulation  sizable line impedance parameters are used and our voltage control policy is implemented in each inverter.}}
					 does not vary much compared to its nominal value \cite{yemimo2020, pogaku2007},
					the active power is proportionally shared among the inverters according to the ratio ${k_{pk}}/{k_{pj}}$.
					The values of  $k_{p,j}$ are chosen inversely  {proportional}
					to the rating of the inverters, where those with high ratings take small values  and vice versa.

				\begin{rmk}\label{rem:Papp}
					\ryo{To explain the approximation used in \eqref{condp1},
				consider the active power relation $P^*_{o,j}:=v^*_{oD,j}i^*_{oD,j}+v^*_{oQ,j}i^*_{oQ,j},\,\forall j\in N$ \cite{pogaku2007}. The power sharing ratio between inverter $j,k\in N$ is given by
				\begin{equation}\label{po-1}
\frac{P^*_{o,j}}{P^*_{o,k}}=\frac{v^*_{oD,j}i^*_{oD,j}+v^*_{oQ,j}i^*_{oQ,j}}{v^*_{oD,k}i^*_{oD,k}+v^*_{oQ,k}i^*_{oQ,k}}
				\end{equation}
				\il{Similarly to} conventional implementations \cite{pogaku2007}, \il{in our voltage control policy in \eqref{acv3}  we choose the nominal voltage as $\mathbf{{e}}V_n$ in DQ coordinates, i.e. its  direct-component is equal to the nominal value $V_n$ while its quadrature-component is chosen to be zero.}
Hence  the voltage direct-components $v^*_{oD,j}$ are approximately $V_n$ while  the voltage quadrature-components $v^*_{oQ,j}$ take values close to zero \cite{yemiisgt2021, pogaku2007}. 
\il{\ic{Therefore,} we have $v^*_{oQ,j}i^*_{oQ,j} \il{\ll} v^*_{oD,j}i^*_{oD,j}$ $\forall j\in N$,
				\ic{thus} \eqref{po-1} can} be approximated as
				\begin{equation*}	
					\frac{P^*_{o,j}}{P^*_{o,k}}=\frac{v^*_{oD,j}}{v^*_{oD,k}} \frac{i^*_{oD,j}}{i^*_{oD,k}}=\frac{v^*_{oD,j}}{v^*_{oD,k}} ~\frac{k_{p,k}}{k_{p,j}} , \hspace{3mm} \forall j,k\in N
				\end{equation*}	
			where the latter equality follows by using \eqref{condp}, hence yielding the power sharing ratio \eqref{condp1}.}
%
		\end{rmk}

			\end{rmk}

\il{Proposition \ref{pshare}} is a statement about the equilibrium point. It can be shown that the equilibrium point is also locally asymptotically stable under \yo{the \irr{commonly used}} assumption of timescale separation between \irr{secondary} control and the inverter/line dynamics, \yo{where the former is much slower than the latter.}
For the analysis below we assume that $\chi$ is updated at a much slower timescale ($100\,\text{ms}$) than the inverter and line dynamics ($1\,\text{ms}$) such that  in this timescale
~\eqref{net},~\eqref{crfinv2},~\eqref{der8},~\eqref{idc2},  ~\eqref{acv1}  
is assumed to have reached \irr{equilibrium;} thus we obtain the linearized static model~\eqref{e13} (see Appendix~\ref{timesep} for its derivation).
\begin{equation} \label{e13}
	\tilde\delta= -(k_Ik_p^{-1}+ F(\delta^*)V_n)^{-1}k_p^{-1}\tilde\chi
\end{equation}
where
\begin{equation} \label{fdelta}
	F(\delta^*)=\underline{\mathbf{e}}^\top Y_2\mathbf{J}^{\top}\mathbf{T}(\delta^*)\mathbf{\underline{e}}
\end{equation}
$Y_2=((R_{c}-\omega_0 L_{c}\mathbf{J})+Y_1^{-1}-\underline{n}_q)^{-1}$,
$Y_1=(G_l-\omega_0C_l\mathbf{J})+(R_{\ell}-\omega_0L_{\ell}\mathbf{J})^{-1}+\mathbf{B}(R_l-\omega_0L_l\mathbf{J})^{-1}\mathbf{B}^{\top}$.
Linearizing~\eqref{sec1} around
$(\chi^*, \delta^*)$ gives
\begin{equation}\label{sec3}
	\dot {\tilde\chi} = -\alpha\mathcal{L}\tilde \chi+\alpha\mathcal{L}k_I\tilde\delta.
\end{equation}
Furthermore, we define the following quantity:
\begin{equation}\label{eq:M}
	\mathcal{M}(\mathbf{\delta^*})=\mathbf{I}_n+k_I(k_Ik_p^{-1} + F(\mathbf{\delta^*})V_n)^{-1}k_p^{-1}
\end{equation}
We now state the {following} stability result. The proof can be found  in Appendix~\ref{prf:stabilityps}.

\begin{thrm}\label{stabilityps} 
	Consider  system
	~\eqref{e13},~\eqref{sec3} and $\mathcal{M}(\mathbf{\delta^*})$ as in~\eqref{eq:M}.
	Suppose 
	$|\delta^*_j|<\pi/2,\,\forall j\in N$, and $k_{p,j},\,k_{I,j}, \forall j\in N$ are 	
	selected such that
	$\tau=k_{I,j}/k_{p,j}\,\forall\, j\in N,$ for some $\tau>0$.
	When $|\delta^*_j|$ are sufficiently small at an equilibrium point of the interconnected system, then all trajectories in~\eqref{e13},~\eqref{sec3} converge to an equilibrium point.
	More precisely,
	  {convergence}
is guaranteed  if
	\begin{equation} \label{matrixcond}
		\norm{\Delta}_2 < K^{-1}\lambda_{n-1}(\hat H)
	\end{equation}
	where
 $\Delta = \mathcal{L}(\mathcal{M}(\mathbf{\delta^*})-\mathcal{M}(\mathbf{0}))$,
 $\mathcal{M}(\mathbf{0})=\mathbf{I}_n+(\mathbf{I}_n + \frac{1}{\tau}F(\mathbf{0})V_n)^{-1}$,
  $\hat H = \mathcal{L}\mathcal{M}(\mathbf{0})$,
 $K = \norm{\Psi^{-1}}_2\norm{\Psi}_2$ is the condition number, where $\Psi$ is the diagonalizing eigenbasis {of $\hat H$}, and
	the eigenvalues of $\hat H$ in descending order {are} $\lambda_1(\hat H), ..., \lambda_{n-1}(\hat H), \lambda_n(\hat H)$ (all eigenvalues of {$\hat H$} are real).
\end{thrm}
\begin{rmk}
	The upper bound in~\eqref{matrixcond} 
		can easily be computed
	as $\mathcal{L}$ and $\mathcal{M}(\mathbf{0})$ are known matrices.
It should also be noted that in the example given in section~\ref{results},
	this	condition is only slightly conservative and is easily satisfied
	(with considerable margin) for all realistic values of~$\mathbf{\delta^*}$.	
\end{rmk}

\ryo{\begin{rmk}
		The requirement on the  angles  $|\delta^*_j|<\pi/2$ in Theorem \eqref{stabilityps}  ensures that the system security constraint is not violated, which is normally satisfied for typical operating points.
\end{rmk}}

	\section{Simulation Results}
\label{results}
In this section, {we illustrate the control policy implementation}, assess the passivity of the inverters
and show via simulations the  performance of the proposed controllers.

\subsection{Implementation of control policy}\label{sec:impl}
We illustrate the implementation of the  control  schemes~\eqref{der8},~\eqref{idc2} and~\eqref{acv1}
in Fig.~\ref{blockfig}. The double loop architecture is shown in~Fig. \ref{twoloops} and the DC voltage system in Fig. \ref{dcloop}. The physical measurements required to implement our controllers are  the filter three-phase  ($abc$) voltage and currents (i.e. $v_o, i, i_o$).
{The  $DQ$ signals used by our controllers are obtained from the $DQ$-transformation of the physical three-phase ($abc$) symmetrical signals.}
The angle droop block uses $\chi$, $I_{oDQ}$ together with~\eqref{crfinv2a} and~\eqref{der8} to compute~$\delta$. The secondary control is in a feedback configuration with the angle droop block and it uses $\delta$ and~\eqref{sec1} to compute $\chi$, which is fed back into the angle droop block. Then, the signals $I_{DQ}$, $V_{oDQ}$, $I_{oDQ}$ and $\delta$ are fed into the double loop voltage control and DC voltage system to compute $m_{DQ}$ using~\eqref{crfinv2b}, \eqref{idc2} and~\eqref{acv1}.
	 Thereafter,  the $abc$ form of $m_{DQ}$  is used in the PWM switching to actuate the inverter electronic switches.

\begin{figure}[ht]
	\centering
	\includegraphics[width=1.\linewidth]{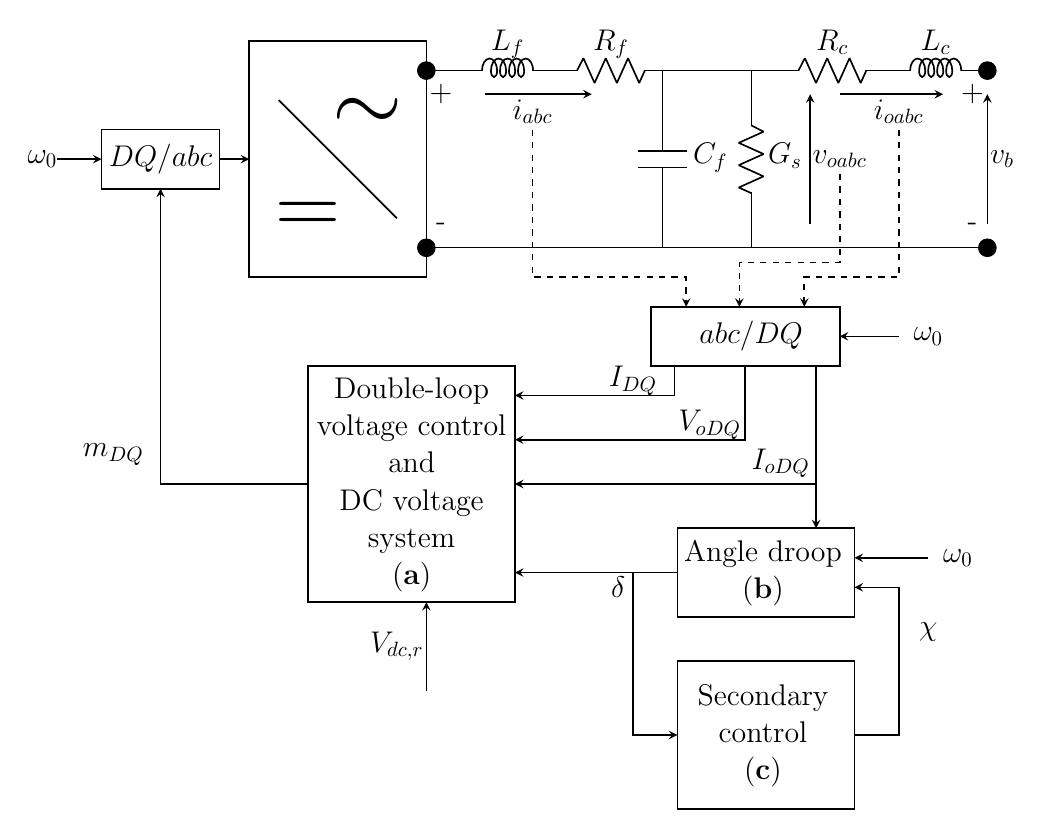}
	\vspace{-8.mm}
	\caption[Block diagram of the proposed control scheme implementation]{Block diagram of the proposed control scheme implementation. \newline
		\textbf{(a)} is \eqref{crfinv2b}, \eqref {idc2}, \eqref{acv1};
		\textbf{(b)} is \eqref{crfinv2b}, \eqref{der8}; and
		\textbf{(c)} is \eqref{sec1}.
	}
\vspace{-3.5mm}
	\label{blockfig}
\end{figure}


\begin{figure}[h!]
	\centering
	\includegraphics[width=1\linewidth]{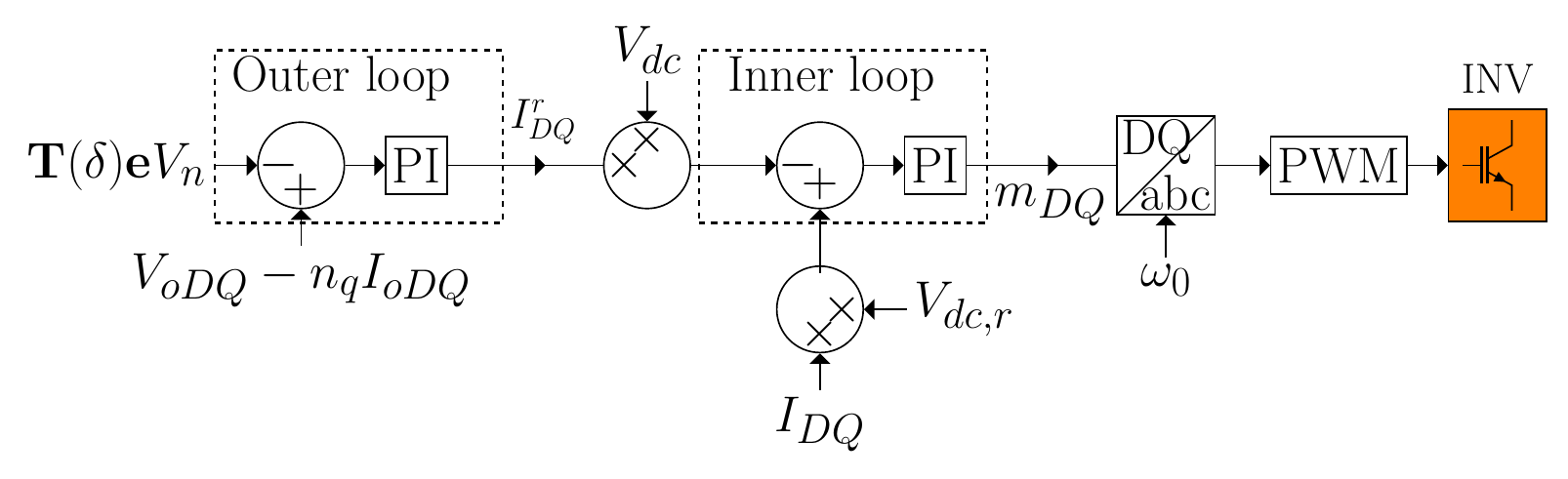}
	\vspace{-8.mm}
	\caption{Block diagram of the proposed double loop voltage control \eqref{acv1}. {{The blocks} denoted by PI represent proportional-integral controllers.}
	}
	\vspace{-3.5mm}
	\label{twoloops}
\end{figure}

\begin{figure}[h!]
	\centering
	\includegraphics[width=1\linewidth]{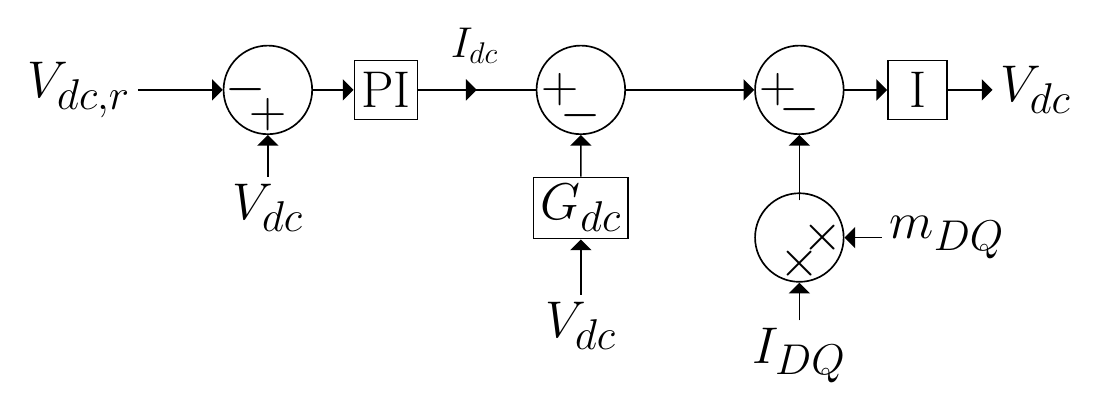}
	\vspace{-4.mm}
	\caption{Block diagram of the DC voltage system~\eqref{crfinv2b},~\eqref{idc2}. {{The blocks} denoted by PI, I respectively represent proportional-integral and integral controllers, respectively}.
	}
	\vspace{-1.5mm}
	\label{dcloop}
\end{figure}

\subsection{Passivity assessment of inverters}
\label{resultspass}
Following our analysis in section~\ref{main}, here we check that appropriate control parameters
associated with the proposed control schemes~\eqref{der8},~\eqref{idc2},~\eqref{acv1} are used to allow the inverters to satisfy the passivity property required  in Assumption~\ref{asspas}.
\begin{figure}[h]
	\centering
	\includegraphics[width=1\linewidth]{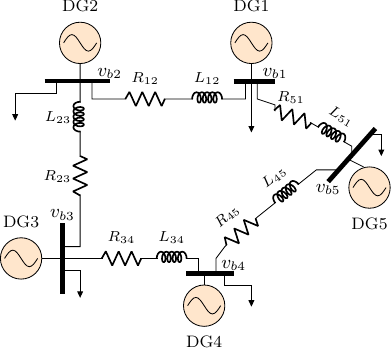}
	\vspace{-3.mm}
	\caption{An autonomous inverter-based microgrid consisting of five grid-forming inverters. The sign $\downarrow$ denotes loads. }
	\vspace{-1.5mm}
	\label{inverterTest}
\end{figure}
We begin by choosing $k_p, \underline{n}_q, c_{p}, c_{I},  \lambda_{P}, \lambda_{I}$.
	Evaluating~\eqref{iner1} at equilibrium shows that small values of $k_{p,j}/k_{I,j}$ allow to have sufficiently small equilibrium angles. Thus we choose small values of $k_{p,j}$,  starting with $k_{p,j}=10^{-3},\,\forall j\in N$.
	To keep  the  steady-state voltage deviation sufficiently small, 	we start with  ${n}_{q,j}=10^{-3},\,\forall j\in N$.
	Following  the standard double  loop design, we choose $c_{p}, c_{I},  \lambda_{P}, \lambda_{I}$  such that  the inner loop is faster than the outer (voltage) loop. In particular, we set the    integral time constant of the inner loop (i.e. $\lambda_{P,j}/\lambda_{I,j}$) such that it is less than  that of the outer loop  (i.e. $c_{p,j}/c_{I,j}$), for all $j\in N$.
	Thus for a start we choose the ratio  $c_{p_j}/c_{I_j}=1/10$,
	$\lambda_{P_j}/\lambda_{I_j}=1/20,\,\forall j \in N$.
	\yo{We note that the only \ir{equilibrium values} 
required  in \eqref{lmii}
	are\footnote{These \irr{appear} 
in $\mathbf{V}^*_{dc}, \mathbf{I}^*_{DQ}, \delta^*, m^*_{DQ}$ 
\irr{in} matrices $\hat A, \hat B$ in Appendix \ref{Amat}.}
	 ${v}^*_{dc}, {i}^*_{DQ,j}, \delta^*_j, m^*_{DQ,j}$.
A good choice is to use a user defined  operating point that
 corresponds to \ir{rated values}  as follows:
  ${v}^*_{dc,j}=V_{dc,r}, {I}^*_{DQ,j}={I}^{d}_{DQ,j}, \delta^*_j=0~\text{rad}, m^*_{DQ,j}=\yl{[0.87~-0.5]^{\top}}$.
This selection is  based on the fact that  ${v}_{dc,j}$   tracks $V_{dc,r}$; the current 
is expected not to exceed its rated value ${I}^{d}_{DQ,j}$;
 the angles should be close to $0~\text{rad}$ for system security;
and the modulating index
specified is the worst case such that the typical requirement $\|{m}^*_{DQ,j}\|_2\leq1$ for linear modulation   is not violated.}
\begin{table}[h]
	\centering
	\caption{Microgrid parameters}\label{tb1}
	\scriptsize
	\begin{tabular}{l|l}
		Description & Value \\
		\hline
		Inverter parameters & $R_{f_j}$=0.1 $\Omega$, $L_{f_j}$=5 $\text{mH}$, $C_{f_j}$=50 $\mu$F, $C_{dc,j}$=10 mS,\\& $R_{c_j}$=0.2 $\Omega$, $L_{c_j}$=2 mH,
		$G_{dc,j}$=10 mS, $G_{s_j}$=3 mS\\
		\hline
		Controller parameters & $\omega_0$ = 2$\pi$(50) rad/s, 
		$v_{dc_{r}}$=10$^{3}$ V,  $V_{n}$=311 V, $\alpha$ = 667,\\& $k_{p,j}$=0.06, $n_{q,j}$=0.078,  $k_{I,j}$=40, $c_{p,j},\Lambda_{P,j}$=1,  \\&$c_{I,j},\Lambda_{I,j}$=10,
		$\lambda_{P,j}$ = $1/v_{dc,r}$,  $\lambda_{I,j}$ = 25$/v_{dc,r}$\\
		\hline
		Loads parameters & $R_{\ell,1}$, $R_{\ell,2}$, $R_{\ell,3}$=20 $\Omega$,
		$R_{\ell,4}$, $R_{\ell,5}$=25 $\Omega$, \\&  $L_{\ell,1}$, $L_{\ell,3}$=30 mH, $L_{\ell,2}$, $L_{\ell,4}$=40 mH, \\& $L_{\ell,5}$=20 mH,  3.0 kW/0.5 kVar at bus 1\\
		\hline
		Switched loads &   2.5 kW at buses 1, 2, 3 \& 4\\
		\hline
		Line parameters & $R_{12}$=0.2 $\Omega$, $R_{45}$ = 0.15 $\Omega$, $R_{23},R_{34},R_{51}$=0.1 $\Omega$, \\ & $L_{12},L_{34}$ = 4 mH,  $L_{23}$=2.8 mH, $L_{45}$=3.5 mH, \\ & $L_{51}$=3 mH, $C_{j}$=0.1 $\mu$H, $G_j$=1 mS\\
		\hline
		Conventional scheme & $k_{p,j}$=0.06$/311$, $n_{q,j}$=0.078$/311$,\\&  $K_{pv}$=5, $K_{iv}$=10, $K_{pi}$=2, $K_{ii}$=15
	\end{tabular}
	\vspace{-3mm}
	\label{tab:params}
\end{table}
\begin{figure}[h]
	\includegraphics[width=1\linewidth]{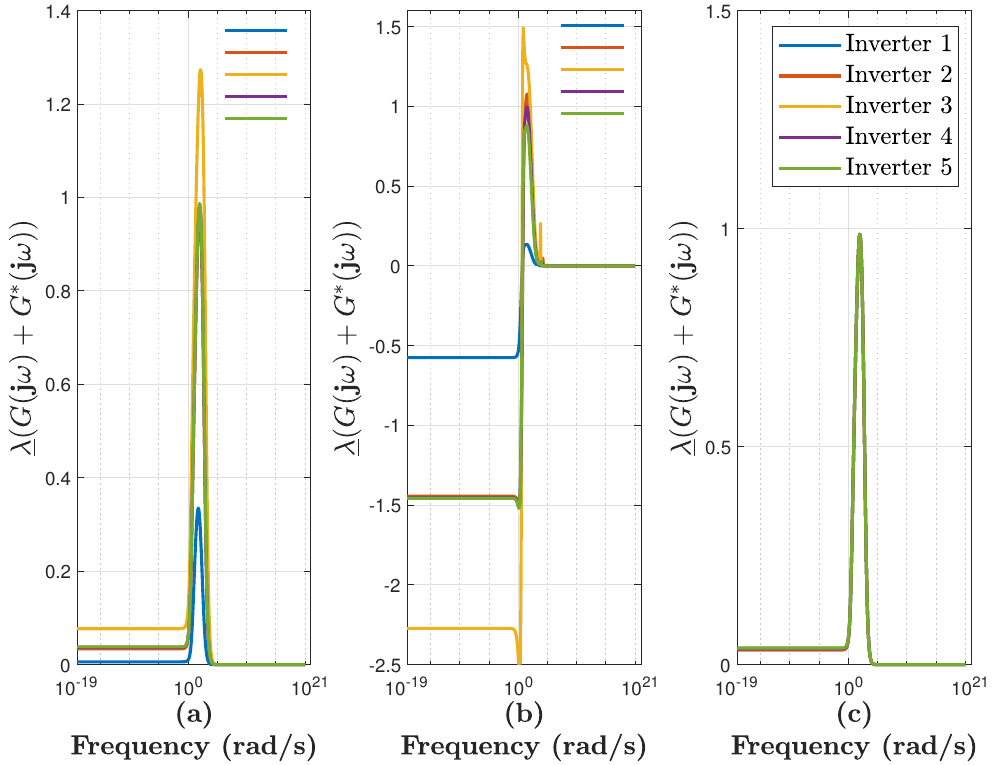}   
	\vspace{-5.mm}
	\caption{ Passivity of grid-forming  inverters showing the minimum eigenvalue {of
			$G(\mathbf{j} \omega)+G^*(\mathbf{j} \omega)$}:
		(a) with {the} proposed control scheme;
		(b) with conventional frequency droop and voltage  scheme~\cite{pogaku2007};
		(c) with {the} proposed control scheme for the example presented in section~\ref{resultssimu}
	}
	\vspace{-3.mm}
	\label{inverterpass}
\end{figure}

We \yo{then}   verify the  passivity property  by searching for the values of  $k_{I,j}$ (typically
within  $0< k_{I,j}\leq100$) that satisfy~\eqref{lmii}.
Further  verification of the passivity property can be performed by using the equilibrium point  obtained from the simulations and minor adjustments can be made to the control parameters to improve performance.
The control gains used for the five-inverter test system (Fig.~\ref{inverterTest}) are given in Table~\ref{tb1}.
 Thus, each inverter satisfies the passivity condition~\eqref{lmii} expressed in the frequency domain (see Remark~\ref{rmk:kyp}). This is  shown in Fig.~\ref{inverterpass}(c) where  the smallest eigenvalue  is {positive} over all frequencies, thus validating that all the inverters in the example are {strictly} passive.

Furthermore, we investigate whether the passivity property  is satisfied with the proposed control scheme~\eqref{der8},~\eqref{idc2},~\eqref{acv1}
on various benchmark examples in~\cite{chandorkar1993, kolluri2017power,  ritwikMajumder2010, ysunGuerrero2017, guerreroMatas2004}  commonly used in the literature to validate control policies for grid forming inverters.
Their inverter parameters are in the range $0.05\,\Omega\leq R_{f_j}\leq 1.5\,\Omega$,
$0.08\,\text{mH}\leq L_{f_j}\leq 8\,\text{mH}$,
$20\,\mu\text{F}\leq C_{f_j}\leq 150\,\mu\text{F}$,
$0.1\,\text{mH}\leq L_{c_j}\leq 30\,\text{mH}$,
$0.03\,\Omega\leq R_{c_j}\leq 2\,\Omega$.
As described above, a user defined operating point corresponding to inverters with rating  $10$-$15$ kVA is used.
We  select suitable control parameters
$0.006\leq k_{p,j}\leq0.06$, $0\leq n_{q,j}\leq0.078$, $1\leq c_{p,j}\leq5$, $10\leq c_{I,j}\leq50$, $10^{-3}\leq \lambda_{P,j}\leq0.1$,  $2.5$$\cdot$$10^{-3}\leq \lambda_{I,j}\leq2.5$.
Then the passivity property is verified by modifying $k_{I}$ such that condition~\eqref{lmii} is satisfied.
Fig.~\ref{inverterpass}(a) shows the passivity result with the proposed scheme, where each plot corresponds to the
benchmark examples in~\cite{chandorkar1993, kolluri2017power,  ritwikMajumder2010, ysunGuerrero2017, guerreroMatas2004}, and this is compared to that with the conventional frequency and voltage control scheme~\cite{pogaku2007} shown in Fig.~\ref{inverterpass}(b).
Fig.~\ref{inverterpass}(a) shows  that the smallest eigenvalue  is positive over all frequencies, thus validating that the inverters in these examples satisfy the passivity property for appropriate values of $k_{I}$, in contrast to those with conventional schemes shown in Fig.~\ref{inverterpass}(b).
We tested the proposed control scheme for other scenarios with realistic operating points for which  $|\delta_j^*|<\pi/2$ rad $\forall j\in N$,
and found the passivity property  to be satisfied for appropriate values of the control gains.

\begin{figure*}[h]
	\centering
	\includegraphics[width=1.\linewidth]{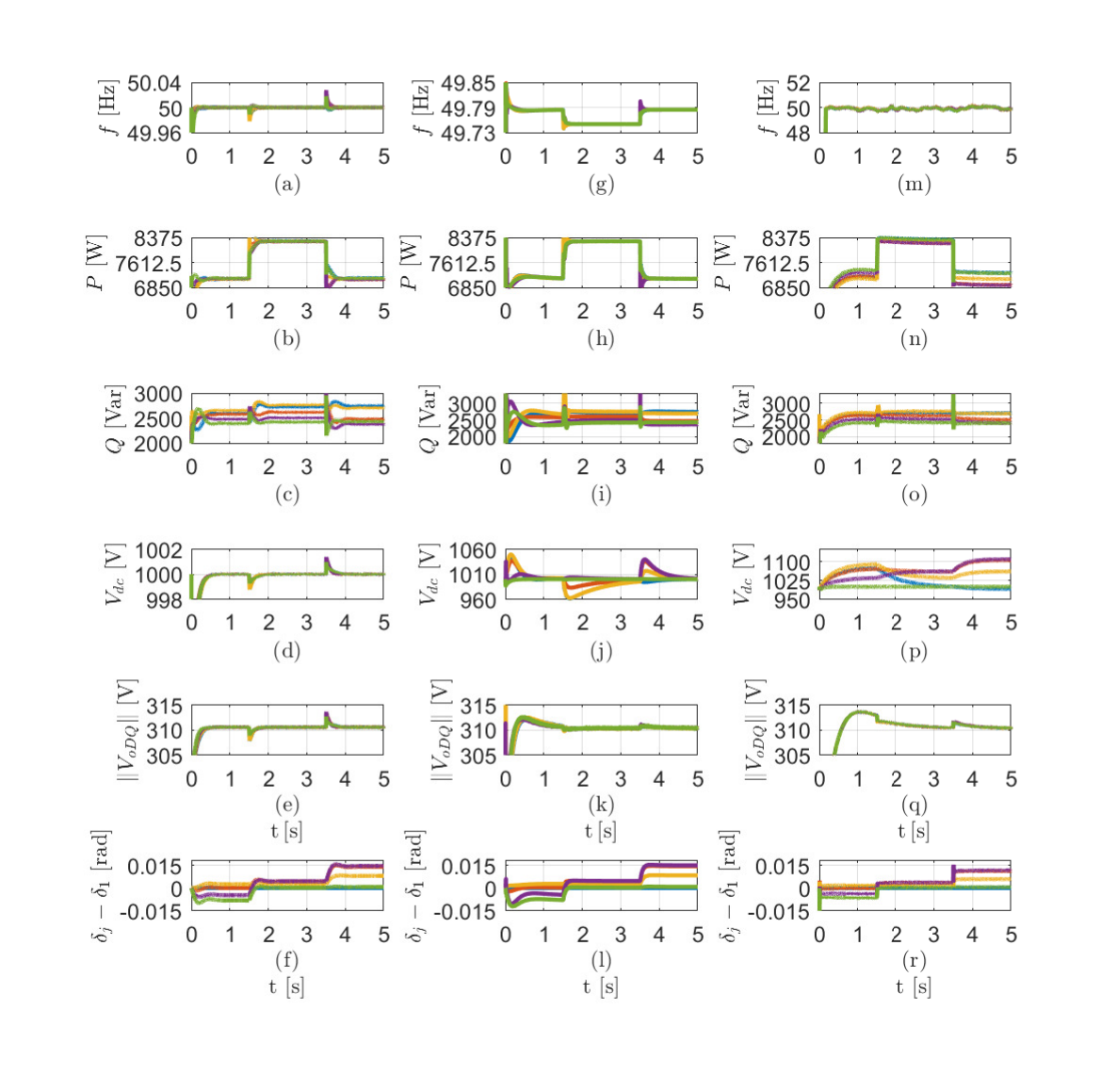}
	\vspace{-18mm}
	\caption{System response with {the} proposed {controller} in (a)--(e) (1st column);
		System response with the conventional frequency droop and voltage  scheme~\cite{pogaku2007} in (g)--(k) (2nd column);
		\yo{System response with the conventional angle droop and voltage  scheme~\cite{majumder2009angle,ysunGuerrero2017} in (m)--(r) (3rd column).}
		\ryo{Inverter 1 'blue', Inverter 2 'red', Inverter 3 'yellow', Inverter 4 'purple', Inverter 5 'green'.}}
	\vspace{-4mm}
	\label{figs}
\end{figure*}


\begin{figure}[h]
	\centering
	\includegraphics[width=1\linewidth]{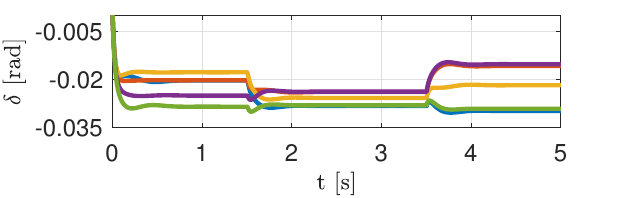}
	\vspace{-6mm}
	\caption{Trajectories of angles  with the proposed controller.		
		\ryo{$\delta_1$ 'blue', $\delta_2$ 'red', $\delta_3$ 'yellow', $\delta_4$ 'purple', $\delta_5$ 'green'.}}
	\vspace{-5mm}
	\label{figsang}
\end{figure} 

\begin{figure*}[h]
	\centering
\includegraphics[width=.95\linewidth]{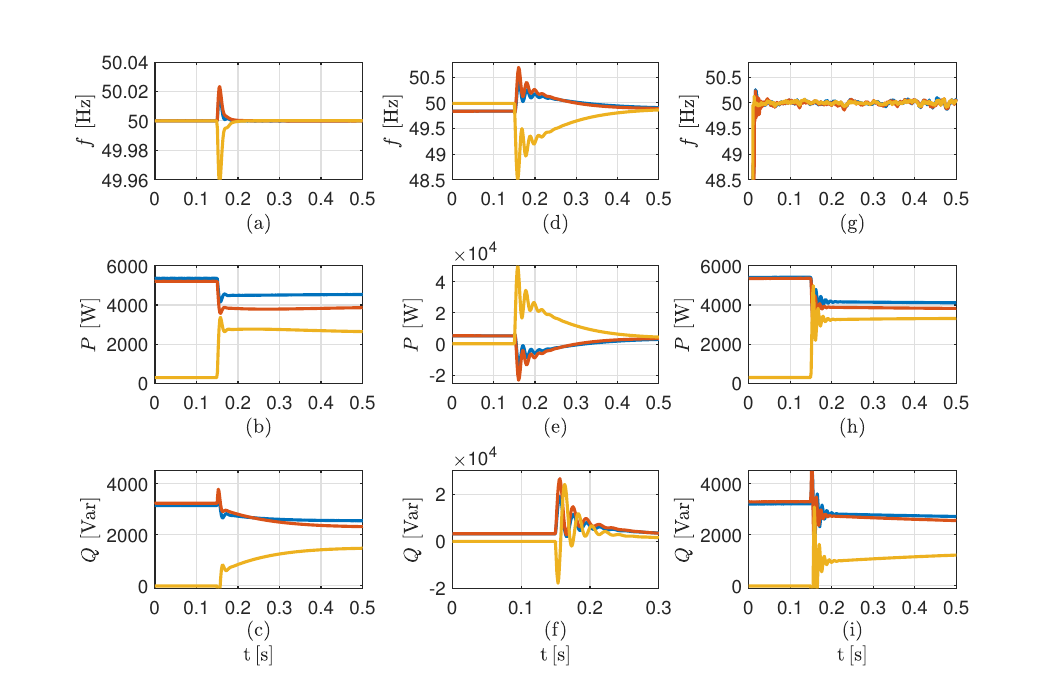}
	\vspace{-5mm}
	\caption{Plug-and-play response with {the proposed controller} in (a)--(c) (1st column);
		Plug-and-play response with conventional frequency droop and voltage  scheme in~\cite{pogaku2007} in (d)--(f) (2nd column);
		\yo{Plug-and-play response with conventional angle droop and voltage  scheme in~\cite{ysunGuerrero2017,majumder2009angle} in (g)--(i) (3rd column);}
		\ryo{Inverter 1 'blue', Inverter 2 'red', Inverter 3 'yellow'.}}
	\vspace{-4mm}
	\label{figss}
\end{figure*} 

\subsection{Numerical simulation}\label{resultssimu}
We show via simulations in MATLAB/Simscape Electrical
the performance of the proposed control policy~\eqref{der8},~\eqref{idc2},~\eqref{acv1},~\eqref{sec1}.
Fig.~\ref{inverterTest} shows the test system of five grid-forming inverters,
and  Table~\ref{tb1} presents the system parameters where the subscript $j=1,\ldots, 5$.
The simulation model is detailed 
and includes the PWM switching of the inverters.
The values of  $k_{I,j}, k_{p,j}$ satisfy the  selections in  Theorem~\ref{stabilityps}.
We compare the performance of the proposed scheme   to that with conventional frequency droop and voltage control~\cite{pogaku2007} \yo{and that with traditional angle droop and voltage control~\cite{majumder2009angle,ysunGuerrero2017}} in the presence of load step changes: a 2.5 kW load is switched on at buses 1 and 3 at $t=1.5\,\text{s}$, and an equivalent load is switching off at buses 2 and 4 at $t=3.5\,\text{s}$.
\yo{To test the robustness of our proposed scheme,} all the switched loads   and that connected throughout the simulation at bus $1$ are nonlinear  constant power loads and these are nominally rated at the reference voltage provided to the inverter.
The resistive-inductive loads  are connected to the corresponding buses throughout the simulation.
The response with the proposed control scheme is shown in
Fig.~\ref{figs}(a)--(f) (1st column), and this is compared to the equivalent response with the conventional frequency droop and voltage control~\cite{pogaku2007} shown in Fig.~\ref{figs}(g)--(l) (2nd column), \yo{and those with the traditional angle droop and voltage control~\cite{majumder2009angle, ysunGuerrero2017} shown in Fig.~\ref{figs}(m)--(r) (3rd column).}
The frequencies synchronize to  $\omega_{0}/2\pi$ Hz, in contrast to the \yo{steady-state frequency deviation} observed with the conventional frequency droop.
The active power sharing \yo{of our control scheme} with communication agrees with Proposition~\ref{pshare}, \yo{compared to some steady state error oberved with the traditional angle droop.}
The proposed control scheme also distributes the reactive power and improves the transient response.
\yo{There is significant improvement in the DC voltage regulation achieved by the proposed control scheme compared to those with the conventional schemes.
Our control policy shows better  transients in the output voltages which are well regulated within the typical requirement $\scriptsize{0.9V_n<\|V_{oDQ,j}\|<1.1V_n}$  contrary to those with the conventional schemes. Fig. \ref{figsang} shows that the angles $\delta$ associated with our control policy are small and are well within $|\delta_j^*|<\pi/2$ rad $\forall j\in N$. This is similarly observed with the angle differences in Fig.~\ref{figs}(f) which  compare to those  with conventional schemes in Fig.~\ref{figs}(l)~\&~\ref{figs}(r).}
The system remains stable  at each operating point with good transient performance  in the presence of  the constant power load disturbances, and hence demonstrates the effectiveness of the proposed control policy.
\yo{We note that  even though matching control (DC voltage based frequency control) \cite{arghirTaouba2018} with the conventional voltage scheme in \cite{pogaku2007} works for a two-inverter system, it could not stabilize the network configuration simulated  in Fig. \ref{inverterTest}. This is caused by the poor regulation of the DC voltage as \irr{also} observed with  the conventional schemes (Fig.~\ref{figs}(j)~\&~\ref{figs}(p)).}

\ryo{Furthermore, we check that Theorem~\ref{stabilityps} is satisfied.
From our simulation we compute the following parameters:
$\|\Delta\|_2=0.0178$,  $\lambda_{n-1}=2.4195$;  $K=1.0057$, and the upper bound as $K^{-1}\lambda_{n-1}=2.4057$. Hence we have $\|\Delta\|_2 < \frac{\lambda_{n-1}}{K}$, which shows that condition~\eqref{matrixcond} is satisfied.
Moreover, the fact that $|\delta^*_j|<\pi/2,\,\forall j\in N$ (Fig. \ref{figsang}),  $\tau=k_{I,j}/k_{p,j}=\frac{40}{0.06}=666.67>0\,\forall\, j\in N$,
and condition~\eqref{matrixcond}  \icl{is}
satisfied  show that   Theorem~\ref{stabilityps} holds for our simulation.}

\subsection{\ir{Plug-and-play} operation}\label{scalability}
We demonstrate the scalability of the proposed scheme as follows. We first simulate the connection of inverter~$1$ and~$2$  in Fig.~\ref{inverterTest}, then inverter~$3$ is {connected}  to bus~$2$.
The response with the proposed control scheme is shown in Fig.~\ref{figss}(a)--(c) (1st column), and this is compared to the equivalent response with the conventional frequency droop and voltage scheme~\cite{pogaku2007} shown in Fig.~\ref{figss}(d)--(f) (2nd column), \yo{and those with the conventional angle droop and voltage scheme~\cite{majumder2009angle, ysunGuerrero2017} shown in Fig.~\ref{figss}(g)--(i) (3rd column).}
In \yo{the three}
cases inverter~$1$ and~$2$ are synchronized before connecting inverter $3$ at $t=0.15\,\text{s}$.
{Note that the secondary controller is not used in the simulation in Fig.~\ref{figss}a--c. As it {involves} 
information from every inverter, it can be activated shortly after connecting the third inverter to achieve active power sharing.}
Fig.~\ref{figss}(a)--(c) shows  that the response of the proposed scheme has much fewer  oscillations and this is without  retuning controller parameters, which demonstrates a plug-and-play capability, \yrr{i.e. new inverters that implement our control policy can be introduced into the network while maintaining its stability, thus allowing to extend to larger networks.}
This is in contrast to \yo{Fig.~\ref{figss}(d)--(f) \& \ref{figss}(g)--(i) where the
conventional control policies
show}  severe oscillations {with} large \irr{overshoots.} 

~\\	
		\section{Conclusion}
		\label{concl}
\yo{We have  proposed  a control architecture for frequency and voltage control which is scalable via a passivity property, allows current limitation via an inner loop, and leads naturally to a distributed secondary controller which achieves \ryo{active} power sharing.
	The frequency controller employs the inverter output current and angle to provide an angle droop-like policy which improves its stability properties and incorporates \ir{a secondary} control policy for which we provide an analytical stability result which takes line conductances into account. The distinctive feature of the voltage control scheme is that it has a double loop structure that uses the DC voltage in the feedback control policy to implement a power-balancing strategy to improve performance. Using passivity analysis,  the stability of the frequency and voltage control was guaranteed  at faster timescales. Simulations	with  detailed  inverter models showed that the control scheme proposed offers good transient performance and scalability.}

\appendices


\section{Proof of Proposition~\ref{pshare}}
\label{prf:pshare}
	At equilibrium the $\dot \chi$ and $\dot \delta$ dynamics  simplify to
$\mathbf{0}_n=  \mathcal{L}k_p \underline{\mathbf{e}}^\top I_{oDQ}^*$, which holds if and only if $k_p \underline{\mathbf{e}}^\top I_{oDQ}^*=\bar\kappa\mathds{1}_n$, $\bar\kappa>0$.
	Thus
	$k_p \underline{\mathbf{e}}^\top I_{oDQ}^*=k_p I_{oD}^*=k_p\,\text{col}(i_{oD,j}^*)=\bar\kappa\mathds{1}_n\Leftrightarrow k_{pj}i_{oD,j}^*=k_{pk}i_{oD,k}^*=\bar\kappa, \forall j,k\in N$ which implies~\eqref{condp}.		
	$\hfill\blacksquare$

\section{Proof of Theorem~\ref{stabilityps}}\label{prf:stabilityps}
Updating~\eqref{sec3} with~\eqref{e13} gives~\eqref{e14}
where $\mathcal{M}(\delta^*)$ is as in~\eqref{eq:M}.
\begin{equation}\label{e14}
	\dot {\tilde \chi} =-\alpha\mathcal{L}\mathcal{M}(\delta^*)\tilde \chi
\end{equation}
{Therefore, the proof of Theorem~\ref{stabilityps}  reduces to 
proving the stability of~\eqref{e14}.
The following lemma is used in the proof of Theorem~\ref{stabilityps}.}

\begin{lem} \label{ml}
	Consider $\mathcal{M}(\delta^*)$ in~\eqref{eq:M} and~\eqref{fdelta}.
	Suppose  $k_{p,j},\,k_{I,j}, \forall j\in N$ are
	selected such that    $\tau=k_{I,j}/k_{p,j}, \forall j\in N$, for some $\tau>0$, and $|\delta^*_j|<\pi/2,\,\forall\,j\in N$.
	Then
	$\mathcal{M}(\delta^*)$
	is strictly diagonally dominant with positive diagonal entries.
	Moreover,
	$\mathcal{M}(\mathbf{0})$
	is symmetric and positive definite.
\end{lem}

\begin{proof}\label{prf:lemstabilityps}
	Given the block diagonal structure of
	 $(G_l-\omega_0C_l\mathbf{J}),
	 (R_{\ell}-\omega_0L_{\ell}\mathbf{J})^{-1},
	(R_{c}-\omega_0 L_{c}\mathbf{J})$, and the fact that
	$\mathbf{B}(R_l-\omega_0L_l\mathbf{J})^{-1}\mathbf{B}^{\top}$
	 is a weighted Laplacian matrix, for sufficiently small  entries of $\underline{n}_q$  the entity $Y_2$ is 	an admittance matrix
	with the structure
	\begin{equation}
		\begin{bmatrix}
			a_{11}\mathbf{I}_2+b_{11}J &-a_{12}\mathbf{I}_2-b_{12}J &\ldots &-a_{1n}\mathbf{I}_2-b_{1n}J\\
			-a_{12}\mathbf{I}_2-b_{12}J &a_{22}\mathbf{I}_2+b_{22}J &\ldots &-a_{2n}\mathbf{I}_2-b_{2n}J
			\\
			\vdots &\vdots &\ddots &\vdots \\-a_{1n}\mathbf{I}_2-b_{nn}J&-a_{2n}\mathbf{I}_2-b_{2n}J &\ldots &a_{nn}\mathbf{I}_2+b_{nn}J
		\end{bmatrix}
	\end{equation}
	where
	$a_{ii}>0, b_{ii}>0,\,a_{ij}>0,\,	b_{ij}>0,\,
	a_{ii}>\sum |a_{ij}|,\, b_{ii}>\sum |b_{ij}|,\, \forall i\neq j,\, i,\, j=1,\ldots, n$  with $n=|N|$.
	For  $\delta^*_i<\pi/2$, let
	$c_{ii}=\cos \delta^*_{i}>0,\, s_{ii}=|\sin \delta^*_{i}|>0, \,  i=1\ldots n$.
	Then from~\eqref{fdelta} we have
	\begin{equation}
		F(\delta^*)=\underline{\mathbf{e}}^\top Y_2\mathbf{J}^{\top}\mathbf{T}(\delta^*)\mathbf{\underline{e}}=\begin{bmatrix}
			F_{11} &F_{12} &\ldots &F_{1n}\\
			F_{21} &F_{22} &\ldots &F_{2n}
			\\
			\vdots &\vdots &\ddots &\vdots \\F_{n1}&F_{n2}&\ldots &F_{nn}
		\end{bmatrix}
	\end{equation}
	where
	$F_{11}=a_{11}s_{11}+b_{11}c_{11}$;
	$F_{1n}=-(a_{1n}s_{nn}+b_{1n}c_{nn})$;
	$F_{22}=a_{22}s_{22}+b_{22}c_{22}$;
	$F_{2n}=-(a_{2n}s_{nn}+b_{2n}c_{nn})$;	
	$F_{nn}=a_{nn}s_{nn}+b_{nn}c_{nn}$;
	$F_{n1}=-(a_{1n}s_{11}+b_{1n}c_{11})$.
	We now check 
	the column diagonal dominance of $F(\delta^*)$, i.e. 
	for every column of $F(\delta^*)$, the magnitude of the diagonal entry in a column is compared to the sum of the magnitudes of all the other (non-diagonal) entries in that column. We have 
	$|F_{21}|+\ldots+|F_{n1}|\leq(|a_{12}|+\ldots+|a_{1n}|)s_{11}+(|b_{12}|+\ldots+|b_{1n}|)c_{11}<a_{11}s_{11}+b_{11}c_{11}=F_{11}$; 	
	similarly
	$|F_{12}|+\ldots+|F_{n2}|<F_{22}$; and
	$|F_{1n}|+\ldots+|F_{(n-1)n}|<F_{nn}$.
	Thus $F(\delta^*)$ is strictly  diagonally dominant with positive diagonal entries.
	Selecting $\tau=k_{I,j}/k_{p,j}$,
	$\tau>0$, gives $k_{I}k_{p}^{-1}=\tau\mathbf{I}_n$ and  $\mathcal{M}(\delta^*)
	=\mathbf{I}_n+(\mathbf{I}_n + \frac{1}{\tau}F(\delta^*)V_n)^{-1}$.
	Note that the
	 strict diagonal dominance of $\mathbf{I}_n + \frac{1}{\tau}F(\delta^*)V_n$
	follows from $F(\delta^*)$.
	Therefore, 	
	${\mathcal{M}}(\delta^*)$ is  strictly diagonally dominant with positive diagonal entries
	since $(\mathbf{I}_n + \frac{1}{\tau}F(\delta^*)V_n)^{-1}$ is  strictly diagonally dominant from the inverse property of diagonally dominant matrices (\cite[Theorem $2.5.11$]{horn1994topics}).
	Observe that $F(\mathbf{0})$ is symmetric
	since 
	$s_{ii}=0,\,c_{ii}=1,\,\forall i\in N$,
	and its positive definiteness follows by noting that $b_{ii}>\sum |b_{ij}|,\, \forall i\neq j,\, i,\, j=1,\ldots, n$.
	Since
	$\mathbf{I}_n+\frac{1}{\tau}F(\mathbf{0})V_n$ is positive definite, then the positive definiteness of
	$\mathcal{M}(\mathbf{0})=\mathbf{I}_n+(\mathbf{I}_n + \frac{1}{\tau}F(\mathbf{0})V_n)^{-1}$
	follows, noting that $(\mathbf{I}_n + \frac{1}{\tau}F(\mathbf{0})V_n)^{-1}$ is positive definite by the inverse property of positive definite matrices~\cite{horn2012matrix}. 	$\hfill\blacksquare$
\end{proof}

We now proceed to prove Theorem~\ref{stabilityps}.
The Laplacian $\mathcal{L}$ is positive semidefinite, having  exactly one zero eigenvalue and all others being strictly positive.
Thus   $\mathcal{L}\mathcal{M}(\mathbf{0})$ also has a single zero eigenvalue and the rest are strictly positive, which can easily be shown by noting that the  eigenvalues of $\mathcal{L}\mathcal{M}(\mathbf{0})$ are the same as the eigenvalues of $(\mathcal{M}(\mathbf{0})^{\frac{1}{2}})\mathcal{L}(\mathcal{M}(\mathbf{0})^{\frac{1}{2}})$, {since $\mathcal{M}(\mathbf{0})$ is positive definite from Lemma~\ref{ml}}.
$\mathcal{L}\mathcal{M}(\delta^*)$  always has a single eigenvalue at the origin since  $\mathcal{M}(\delta^*)$ is strictly diagonally dominant from Lemma~\ref{ml}. Hence, since the eigenvalues of a matrix vary continuously with its parameters~\cite{zedek1965continuity} and  $\mathcal{M}(\delta^*)$ varies continuously with $\delta^*$, there exists some sufficiently small values of $\delta^*$ such that the eigenvalues of $\mathcal{L}\mathcal{M}(\delta^*)$ are non-negative.
Therefore $\mathcal{L}\mathcal{M}(\delta^*)$ has a single eigenvalue at the origin with all other eigenvalues strictly positive when $\delta^*$ is sufficiently small, and hence all trajectories of~\eqref{e14} converge to an equilibrium point.
From an application of the Bauer-Fike theorem on $\Delta+ \hat H$ we derive a bound on $\delta^*$ as follows.
We note  that $\mathcal{L}\mathcal{M}(\mathbf{\delta^*}) = \mathcal{L}\mathcal{M}(\mathbf{0}) + \Delta =\hat H + \Delta$ with $\hat H = \mathcal{L}\mathcal{M}(\mathbf{0})$.
 Since both $\mathcal{L}$ and $\mathcal{M}(\mathbf{0})$ are symmetric matrices and $\mathcal{M}(\mathbf{0})$ is positive definite,
  $\hat H$ is diagonalizable {by means of} its eigenbasis $\Psi$ and a diagonal eigenvalue matrix $\Lambda$ such that $\hat H = \Psi\Lambda \Psi^{-1}$.
We apply the Bauer-Fike theorem to this matrix~\cite{bauer1960norms}, which states that for each eigenvalue $z$ of $\hat H + \Delta$ there is a corresponding eigenvalue $z_{\hat H}$ of $\hat H$ such that:
 \begin{equation}\label{eq:bauerfike}
 	|z - z_{\hat H}| \leq K \norm{\Delta}_2
 \end{equation}
 where $K$ is as defined in the Theorem statement. Both $\mathcal{L}$ and $\mathcal{L}\mathcal{M}(\mathbf{0})$ have exactly one zero eigenvalue, as already shown  above.
 Suppose the second smallest eigenvalue of $\hat H$, $\lambda_{n-1}(\hat H)$, satisfies condition \eqref{matrixcond}.
 Then the second smallest eigenvalue of $\mathcal{L}\mathcal{M}(\mathbf{\delta^*})$ is strictly positive and hence
 all other eigenvalues of $\mathcal{L}\mathcal{M}(\mathbf{\delta^*})$ are strictly positive. $\hfill\blacksquare$

\section{Linearized Static Model }
\label{timesep}

{In this section we derive the linearized static model~\eqref{e13}.}
	Given the timescale separation considered, the time derivatives in
{the} linearized
~\eqref{net}--\eqref{der8},~\eqref{idc2},~\eqref{acv1} {are set}
to zero.   Hence, from the linearized~\eqref{net},~\eqref{netc} we get
\begin{equation}\label{e4}
	\tilde{V}_{bDQ}=Y_1^{-1} \tilde I_{oDQ}
\end{equation}
where
$Y_1=(G_l-\omega_0C_l\mathbf{J})+(R_{\ell}-\omega_0L_{\ell}\mathbf{J})^{-1}+\mathbf{B}(R_l-\omega_0L_l\mathbf{J})^{-1}\mathbf{B}^{\top}$,	and from the linearized
~\eqref{crfinv2},~\eqref{acv1} \yo{and setting their time derivatives to zero} we have
\begin{equation} \label{e6}
	(R_{c}-\omega_0 L_{c}\mathbf{J})\tilde I_{oDQ}= \tilde V_{oDQ}-\tilde V_{bDQ},
\end{equation}	
\begin{equation} \label{e6e}
	\tilde V_{oDQ}=\mathbf{J}^{\top}\mathbf{T}(\delta^*)\mathbf{\underline{e}}V_n\tilde\delta+\underline{n}_q\tilde I_{oDQ}.
\end{equation}
Using \eqref{e4} and substituting \eqref{e6e} for $\tilde V_{oDQ}$ in \eqref{e6} gives
\begin{equation} \label{e10}
	\begin{split}
		\tilde I_{oDQ}=Y_2\mathbf{J}^{\top}\mathbf{T}(\delta^*)\mathbf{\underline{e}}V_n\tilde\delta
	\end{split}
\end{equation}
where	$Y_2=((R_{c}-\omega_0 L_{c}\mathbf{J})+Y_1^{-1}-\underline{n}_q)^{-1}$.
{Substituting~\eqref{der8} into~\eqref{crfinv2a} gives $\dot{\delta}=-k_{I}\delta-k_{p}\underline{\mathbf{e}}^{\top}{I}_{oDQ} -  \chi$.
Linearizing this and setting its time derivative to zero yields $k_{I}\tilde \delta=-k_{p}\underline{\mathbf{e}}^{\top}\tilde{I}_{oDQ} - \tilde \chi$, which then becomes~\eqref{e13} by substituting~\eqref{e10}}.

\section{Definition of Parameters}
\label{Amat}
In this section we define the matrices in the linearized model \eqref{linmodell} and LMI \eqref{lmii}. {In particular, we have} 
\newline
\hspace{-0mm}$A=\Gamma^{-1}\hat{A}$,
$B=\Gamma^{-1}\hat{B}$,
$B_u=\Gamma^{-1}\hat{B}_u$,\\
$\Gamma=\text{blkdiag}(\ryo{\mathbf{I}_{n}, C_{dc}, L_{f}, C_f, L_c, \mathbf{I}_{2n},  \mathbf{I}_{2n}})$,\\
$\hat {B}=\begin{bmatrix}
	\mathbf{0}_{2n\times 2n} &-\frac{1}{2}\mathbf{I}^{*}_{DQ} &\frac{1}{2}\mathbf{V}^{*}_{dc} &\mathbf{0}_{2n\times8n}
\end{bmatrix}^{\top}$, $D_u=\mathbf{0}$,\\
$\hat{B}^{\top}_u=C=[\mathbf{0}_{2n\times7n}\,\, \mathbf{I}_{2n}\,\, \mathbf{0}_{2n\times4n}]$, $C_{\delta}=[\mathbf{I}_{n}\,\, \mathbf{0}_{n\times 12n}]$.\\
Let $\hat{G}_{dc}=G_{dc}+\Lambda_{P}$,
$Z_f=R_f-\omega_{0}L_f\mathbf{J}$,
$Z_s=G_s-\omega_{0}C_f\mathbf{J}$,
$Z_c=R_c-\omega_{0}L_c\mathbf{J}$, then
\begin{equation*}
	\begin{split}
		\hat{A}(1n,:)&=\mathbf{0}_{n\times13n}\hspace{6mm}
		\hat{A}(2n,:)=[\mathbf{0}_{n\times2n}\,\, \mathbf{I}_n \,\, \mathbf{0}_{n\times 10n}]\\
		\hat{A}(3n,:)&=\begin{bmatrix}\mathbf{0}_{n\times n}& -\Lambda_{I} & -\hat{G}_{dc} &-\frac{1}{2}m^{*\top}_{DQ} &\mathbf{0}_{n\times 8n}]\end{bmatrix}\\
		\hat{A}(4n:5n,:)&=\begin{bmatrix}
			\mathbf{0}_{2n\times2n} &\frac{1}{2}m^*_{DQ} &-Z_f &-\mathbf{I}_{2n} &\mathbf{0}_{2n\times6n}
		\end{bmatrix}\\
		\hat{A}(6n:7n,:)&=\begin{bmatrix}
			\mathbf{0}_{2n\times3n} &\mathbf{I}_{2n} &-Z_s -\mathbf{I}_{2n} &\mathbf{0}_{2n\times4n}
		\end{bmatrix}\\	
		\hat{A}(8n:9n,:)&=\begin{bmatrix}
			\mathbf{0}_{2n\times5n} &\mathbf{I}_{2n} &-Z_c  &\mathbf{0}_{2n\times4n}
		\end{bmatrix}\\
		\hat{A}(10n:11n,:)&=\begin{bmatrix}
			\mathbf{J}\mathbf{T}(\delta^*)\underline{\mathbf{e}}V_n  &
			\mathbf{0}_{2n\times4n} &\mathbf{I}_{2n} &-\underline{n}_q  &\mathbf{0}_{2n\times4n}
		\end{bmatrix}\\
		\hat{A}(12n:13n,:)&=[
		c_p\mathbf{V}^*_{dc}\mathbf{J}\mathbf{T}(\delta^*)\underline{\mathbf{e}}V_n  \,\,
		\mathbf{0}_{2n\times n} \,\,\mathbf{I}^*_{DQ} \,\,\mathbf{V}^*_{dc} \\&\hspace{6mm}c_p\mathbf{V}^*_{dc} \,\,-c_p\mathbf{V}^*_{dc}\underline{n}_q  \,\,c_I\mathbf{V}^*_{dc} \,\,\mathbf{0}_{2n\times2n}
		].
	\end{split}
\end{equation*}
\begin{equation*}
	\begin{split}
		\Sigma=&P(A-C_{\delta}^{\top}k_p\underline{\mathbf{e}}^{\top}C-C_{\delta}^{\top}k_IC_{\delta}-B\hat K)\\&+ (A-C_{\delta}^{\top}k_p\underline{\mathbf{e}}^{\top}C-C_{\delta}^{\top}k_IC_{\delta}-B\hat K)^{\top}P\\
		\hat K=&\begin{bmatrix}
			k_1& \mathbf{0}_{2n\times n}& k_2& k_3& k_4& k_5& k_6& k_7		\end{bmatrix}
	\end{split}
\end{equation*}
$k_1=\lambda_pc_p\mathbf{V}^*_{dc}\mathbf{J}\mathbf{T}(\delta^*)\underline{\mathbf e}V_n$,
$k_2=-\lambda_p\mathbf{I}^*_{DQ}$,
$k_3=-\lambda_p\mathbf{V}^*_{dc}$,
$k_4=\lambda_pc_p\mathbf{V}^*_{dc}$,
$k_5=-\lambda_pc_p\mathbf{V}^*_{dc}\underline{n}_q$,
$k_6=\lambda_pc_I\mathbf{V}^*_{dc}$,
$k_7=\lambda_I$.

\bibliographystyle{IEEEtran}
\bibliography{bib_inverterbased}

	\end{document}